\documentclass[11pt,reqno]{amsart}
\usepackage{amsmath} 
\usepackage{amssymb}
\usepackage{dsfont}

\def\squarebox#1{\hbox to #1{\hfill\vbox to #1{\vfill}}}

\theoremstyle{plain}

\newtheorem{Thm}{Theorem}
   
\newtheorem{cor}{Corollary}

\newtheorem{lem}{Lemma}

\newcommand{\Div}{\textrm{div}}
\newcommand{\R}{\mathbb{R}}
 
\newcommand{\B}{\dot{\mathcal{H}}_1({\R}^n)}
\newcommand{\U}{{\mathcal U}}
\newcommand{\CI}{{\mathcal C}^{\infty}_{0}({\R}^{n}) }

\newcommand{\half}{\frac{1}{2}}

\def\epsilon{\varepsilon}
\def\phi {\varphi}

\newtheorem{rem}{Remark}
\newtheorem{prop}{Proposition}

\numberwithin{equation}{section}

\begin{document}

\title[Global Strichartz estimates ]
{ Global Strichartz estimates for the wave equation with a time-periodic non-trapping metric}

\author[Y. Kian]{Yavar Kian}

\address {Universit\'e Bordeaux I, Institut de Math\'ematiques de Bordeaux,  351, Cours de la Lib\'eration, 33405  Talence, France}

\email{Yavar.Kian@math.u-bordeaux1.fr}
\maketitle
\begin{abstract}
We obtain  global Strichartz estimates for the solution $u$ of the wave equation\\
 $\partial_t^2 u-\Div_x(a(t,x)\nabla_xu)=0$ with time-periodic metric $a(t,x)$ equal to $1$ outside a compact set with respect to $x$. We assume $a(t,x)$ is a non-trapping perturbation and moreover, we suppose that there are no resonances $z_j\in\mathbb{C}$ with $\vert z_j\vert\geq1$. 
\end{abstract}

\section{Introduction}
\renewcommand{\theequation}{\arabic{section}.\arabic{equation}}
\setcounter{equation}{0}

Consider the Cauchy problem
\begin{equation} \label{eq:2}  \left\{\begin{array}{c}
u_{tt}-\Div_{x}(a(t,x)\nabla_{x}u)=0,\ \ (t,x)\in{\R}^{n+1},\\
(u,u_{t})(s,x)=(f_{1}(x),f_{2}(x))=f(x),\ \ x\in{\R}^n,\end{array}\right.\end{equation}
where the perturbation $a(t,x)\in C^\infty({\R}^{n+1})$ is a scalar function which satisfies the conditions:
\begin{equation}\label{eq=1}\begin{array}{l}(i) \ C_0\geq a(t,x)\geq c_0>0,\    (t,x)\in{\R}^{n+1},\\
(ii)\ \textrm{ there exists }\rho>0\textrm{ such that }a(t,x)=1\textrm{ for }\vert x\vert\geq\rho,\\
(iii)\textrm{ there exists }T>0\textrm{ such that }a(t+T,x)=a(t,x),\    (t,x)\in{\R}^{n+1}.\\
\end{array}\end{equation}

Throughout this paper we assume that $n\geq3$ is odd. Let  $\dot{H}^\gamma({\R}^n)=\Lambda^{-\gamma} (L^2({\R}^n))$ be the homogeneous Sobolev spaces, where $\Lambda=\sqrt{-\Delta_x}$ is determined by the Laplacian in ${\R}^n$. The solution of (1.1) is given by the propagator
\[\mathcal{U}(t,s):\dot{\mathcal{H}}_\gamma({\R}^n)\ni(f_1,f_2)=f\mapsto \mathcal{U}(t,s)f=(u,u_t)(t,x)\in\dot{\mathcal{H}}_\gamma({\R}^n)\]
where $\dot{\mathcal{H}}_\gamma({\R}^n)=\dot{H}^\gamma({\R}^n)\times \dot{H}^{\gamma-1}({\R}^n)$.

We  say that the numbers $2\leq p,q\leq\infty$, $\gamma>0$ are admissible for the free wave equation $\partial_t^2u-\Delta_xu=0$ if for $f\in\dot{\mathcal{H}}_\gamma({\R}^n)$ the solution $u(t,x)$ of (1.1), with $a=1$ and $s=0$, satisfies the estimate
\[\Vert u\Vert_{L^p({\R}^+,L^q({R}^n))}+ \Vert u(t)\Vert_{\dot{H}^\gamma}+\Vert \partial_t(u)(t)\Vert_{\dot{H}^{\gamma-1}}\leq C(p,q,\rho,T)(\Vert f_1\Vert_{\dot{H}^\gamma}+\Vert f_2\Vert_{\dot{H}^{\gamma-1}})\]
with $C(p,q,\rho,T)>0$ independent on $t$. It is well known (see for instance, [14]) that  $(p,q,\gamma)$ are admissible if the following condition holds
\begin{equation}\label{eq:1}\frac{1}{p}+\frac{n}{q}=\frac{n}{2}-\gamma,\quad\frac{1}{p}\leq \left(\frac{n-1}{2}\right)\left(\half-\frac{1}{q}\right)\end{equation}
with $(p,q,\gamma)\neq(2,\infty,1)$ when $n=3$.
Our purpose in this paper is to establish Strichartz estimates for the problem (1.1) assuming the perturbation $a(t,x)$ non-trapping. More precisely, consider the null bicharacteristics $(t(\sigma),x(\sigma),\tau(\sigma),\xi(\sigma))$ of the principal symbol $\tau^2-a(t,x)\vert\xi\vert^2$ of $\partial^2_t-\Div_x(a\nabla_x)$ satisfying
\[t(0)=t_0,\vert x(0)\vert\leq\rho,\quad \tau^2(\sigma)=a(t(\sigma),x(\sigma))\vert\xi(\sigma)\vert^2.\]
We introduce the following condition

$\ \ $(H1) We say that the metric $a(t,x)$ is non-trapping if for $R>\rho$ there exists $S_R>0$ such that $\vert x(\sigma)\vert>R$ for $\vert \sigma\vert\geq S_R$. 

The non-trapping condition (H1) is necessary for the Strichartz estimates since for some trapping perturbations we may have solutions with exponentially increasing local energy (see [7]). On the other hand, even for non-trapping periodic perturbation some parametric resonances could lead to solutions with exponentially growing local energy (see [6] for the case of time-dependent potentials). To exclude the existence of such solutions we must impose a second hypothesis.

Let $U_0(t)=e^{iGt}$ be the unitary group on ${\B}$ related to the Cauchy problem (1.1) for the free wave equation  ($a=1$ and $\tau=0)$. For $b\geq\rho$ denote by $P^b_+$ (resp $P^b_-$)  the orthogonal projection on the orthogonal complements of the Lax-Phillips spaces 
\[D^b_\pm=\{f\in{\B}\ :\ (U_0(t)f)_1(x)=0\  \textrm{ for } \ \vert x\vert<\pm t+b\}.\]
Set $Z^b(t,s)=P^b_+\mathcal{U}(t,s)P^b_-$. Then the resonances of the problem (1.1) coincide with the eigenvalues of the operator $Z^b(T,0)$ and the condition (H1) guarantees that the the spectrum $\sigma(Z^b(T,0))$ of $Z^b(T,0)$ is formed by eigenvalues $z_j\in\mathbb{C}$ with finite multiplicities. Moreover, these eigenvalues are independent on the choice of $b\geq\rho$ (see for more details [18] for time-periodic potentials at moving obstacles). Our second condition is

$\ \ $(H2)$\quad\quad\quad\quad\quad\quad\quad\quad\quad\quad$ $\sigma(Z^b(T,0))\cap\{z\in\mathbb{C}\ :\ \vert z\vert\geq1\}=\varnothing.$

Assuming (H1) and (H2), we establish an exponential decay of local energy similar to the results for time-dependent potentials and non-trapping moving obstacles (see [3], [8], [18]).

Our main result is the following
\begin{Thm}
Let $a(t,x)$ be a metric for which the conditions (H1) and (H2) are satisfied.  Assume that
 $2\leq p,q\leq+\infty$ satisfy conditions 
\begin{equation} \label{eq:2}
 \begin{array}{l}
(i) \textrm{ if } n=3,\ \ q>6 \ \ \textrm{ and }\ \frac{1}{p}=\frac{n(q-2)}{2q}-1,\\
(ii) \textrm{ if } n\geq5\textrm{ is odd},\ \ \frac{2n}{n-2}<q<\frac{2n}{n-3} \ \ \textrm{ and }\ \frac{1}{p}=\frac{n(q-2)}{2q}-1.
\end{array}\end{equation}
Then for   $u$ the solution of (1.1) with $s=0$ we have for all $t>0$ the  estimate
\begin{equation} \label{eq:2}\Vert u\Vert_{L^p({\R}^+,L^q({R}^n))}+ \Vert u(t)\Vert_{\dot{H}^1}+\Vert \partial_t(u)(t)\Vert_{L^2({\R}^n)}\leq C(p,q,\rho,T)(\Vert f_1\Vert_{\dot{H}^1}+\Vert f_2\Vert_{L^2({\R}^n)}).\end{equation}
\end{Thm}

\begin{rem}
 The conditions (1.3) for the free wave equations imply for $\gamma=1$ and $2\leq p,q\leq+\infty$ that (see Section 6)
\[\frac{1}{p}=\frac{n(q-2)}{2q}-1, \frac{2n}{n-2}\leq q\leq\frac{2n}{n-3}, \textrm{ for } n\geq5\]
and
\[\frac{1}{p}=\frac{n(q-2)}{2q}-1, q\geq6, \textrm{ for } n=3.\]

This means that our conditions  (1.4) are stronger than (1.3) and the assumption of Theorem 1 makes possible to apply the estimates (1.2) for the free wave equation. 
\end{rem}
\begin{rem}
 Since the differential  operator $\Div_x(a(t,x)\nabla.)$ is time-dependent we cannot apply the argument of Tao [14], to obtain  inhomogeneous Strichartz estimates (see for example [14]). 
\end{rem}
\begin{rem}
Let the metric $(a_{ij}(t,x))_{1\leq i,j\leq n}$ be such that for all $i,j=1\cdots n$ we have
\[\begin{array}{l}
\displaystyle(i)\ \textrm{ there exists }\rho>0\textrm{ such that }a_{ij}(t,x)=\delta_{ij},\textrm{ for }\vert x\vert\geq\rho,\textrm{ with $\delta_{ij}=0$ for $i\neq j$ and $\delta_{ii}=1$},\\
\displaystyle(ii)\textrm{ there exists }T>0\textrm{ such that }a_{ij}(t+T,x)=a_{ij}(t,x),\    (t,x)\in{\R}^{n+1},\\
\displaystyle(iii)a_{ij}(t,x)=a_{ji}(t,x),  (t,x)\in{\R}^{n+1},\\
\displaystyle(iv)\textrm{ there exist } C_0>c_0>0 \textrm{ such that }C_0\vert\xi\vert^2\geq\sum_{i,j=1}^na_{ij}(t,x)\xi_i\xi_j\geq c_0\vert\xi\vert^2 ,\ \  (t,x)\in{\R}^{1+n},\ \xi\in{\R}^n.\\
\end{array}\]

If we replace $a(t,x)$ in (1.1) we get the following problem
\begin{equation} \label{eq:2}  \left\{\begin{array}{c}
\displaystyle u_{tt}-\sum_{i,j=1}^n\frac{\partial}{\partial x_i}\left(a_{ij}(t,x)\frac{\partial}{\partial x_j}u\right)=0,\ \ (t,x)\in{\R}^{n+1},\\
\ \\
(u,u_{t})(s,x)=(f_{1}(x),f_{2}(x))=f(x),\ \ x\in{\R}^n.\end{array}\right.\end{equation}

Repeating  the  argument  for   (1.1) we can prove that estimates (1.5) are  true for the solution $u$ of the problem (1.6) if for the trajectories of the symbol $\tau^2-\sum_{i,j=1}^na_{ij}(t,x)\xi_i\xi_j$ and the corresponding operator $Z^b(T,0)$, (H1) and (H2) are fulfilled and if $n,p,q$ satisfied (1.4).
\end{rem}
Strichartz estimates for the wave equation with time-periodic potential have been established in [19]. In the proof of [19] the $L^2$ -integrability of the local energy plays a crucial rule (see also [1]). In our case we obtain also a $L^2$-integrability of the local energy assuming (H1) and (H2) fulfilled. However, in contrast to  the wave equation with time-periodic potential examined in [19], to obtain  Strichartz estimates (1.5) we need  local Strichartz  estimates for the perturbed equation. For this purpose, we construct a local smooth approximation of the  Cauchy problem (1.1) by Fourier integral operators and we obtain uniform estimates for their phases and amplitudes. By using this approximation, we apply the result of Kapitanski [12] to get  local Strichartz estimates. Finally, combining local Strichartz estimates and the $L^2$-integrability of the local energy, we deduce the Strichartz estimates (1.5) by applying the idea of Burq [1] to decompose
\[u=(1-\chi)u+\chi u,\]
where $\chi\in{\CI}$ and $\chi=1$ for $\vert x\vert\leq\rho$. To  estimate $\chi u$ we exploit  local Strichartz estimates and to deal with $(1-\chi)u$ we apply the global Strichartz estimate for the free wave equation. In Section 8 we present some examples of function $a(t,x)$ for which (H1) and (H2) are fulfilled.

\vspace{5mm}

\textbf{Acknowledgements}. The author would like to thank Vesselin Petkov for his precious help during the preparation of this work, Jean-Fran\c{c}ois Bony  for  his remark and the referee for his suggestions.

\section{Exponential decay  of local energy}
\renewcommand{\theequation}{\arabic{section}.\arabic{equation}}
\setcounter{equation}{0}

Throughout this section we will establish that condition (H2) implies the exponential decay of the local energy. In the same time we will recall some properties of the operator $Z^b(t,s)$ and show that for $t$ sufficiently large $Z^b(t,0)$ is compact in ${\B}$. The properties of $Z^b(t,s)$ are proved for the wave equation with time dependent potential in chapter V of [18]. The same proofs work for the wave equation with a time dependent metric satisfying conditions (1.2). 

We  start with some general properties of $\mathcal{U}(t,s)$ and $Z^b(t,s)$. These properties are established in [18] for the wave equation with time dependent potential. The same proof works for the problem (1.1).

\begin{prop}
 Let $b>\rho$, $\tau,\ t\in{\R}$. Then\newline
(i)$\mathcal{U}(t+T,s+T)=\mathcal{U}(t,s)$\\
(ii) $\quad \mathcal{U}(t,s)=U_0(t-s)f,\quad   f\in D^b_+$, if $s\leq t$.\newline
(iii) $ \quad \mathcal{U}(s,t)f=U_0(s-t)f,\quad  f\in D^b_-$, $s\leq t$.\newline
(iv) $\quad \mathcal{U}(t,s)((D^b_-)^\bot)\subset (D^b_-)^\bot$.\newline
(v) For all $s_1,t_1,s_2\in{\R}$ such that $s_1\geqslant t_1\geqslant s_2$ we have\\
 $Z^b(s_1,t_1)Z^b(t_1,s_2)=Z^b(s_1,s_2)$. \newline
(v)$\quad Z^b(t+T,s+T)=Z^b(t,s),\quad   t,s\in{\R}$.
\end{prop}

 Since $a(t,x)$ is  non-trapping, the results of the propagation of singularities imply (see [17]) the following
\begin{prop}
For all $R>0$, there exists $T(R)>0$ such that for all $s\in{\R}$ and for all $f\in {\B}$ satisfying $supp(f)\subset B_R$ we have
\begin{equation}\label{eq:hyp1}
(t,x)\longmapsto (\mathcal{U}(t,s)f)(x)\in (C^\infty(]T(R)+s,+\infty[\times B_R))
\end{equation}
where $B_r=\{x\in{\R}^n : \vert x\vert\leq r\},\quad   r>0$.
\end{prop}
We deduce from Proposition 1 and Proposition 2 that the operator  $Z^b(t,0)$ is compact for  $t>0$ large enough. More precisely, we have the following
\begin{prop}
 Let $b>\rho$, and let $n\geq 3$ be odd. Then, for all $t\geq 4b+T_{4b} $, the operator $Z^b(t,0)$ is compact in ${\B}$
\end{prop}
\begin{proof}
Set $M(t,s)=\mathcal{U}(t,s)-U_0(t-s)$ and write
\[\begin{array}{lll}Z^b(t,0)&=&P^b_+M(t,t-2b)\mathcal{U}(t-2b,2b)M(2b,0)P^b_-+P_+^bM(t,t-2b)\mathcal{U}(t-2b,2b)U_0(2b)P^b_-\\
   \ &\ &+P_+^bU_0(2b)\mathcal{U}(t-2b,0)P^b_-\end{array}.\]
Choose $\chi\in{\CI}$ so that $\chi=1$ for $\vert x\vert\leq3b$ and $\chi=0$ for $\vert x\vert\geq4b$. Taking into account the finite speed of propagation, it is easy to see that
\[M(s,s-2b)=\chi M(s,s-2b)\chi.\]
We find
\[\begin{array}{lll}Z^b(t,0)&=&P^b_+M(t,t-2b)\chi \mathcal{U}(t-2b,2b)\chi M(2b,0)P^b_-+P_+^bM(t,t-2b)\chi \mathcal{U}(t-2b,2b)U_0(2b)P^b_-\\
   \ &\ &+P_+^bU_0(2b)\mathcal{U}(t-2b,0)P^b_-\end{array}.\]
It follows from the properties (ii) and (iv) of Proposition 1 that for $t>4b$ we have
\[\mathcal{U}(t-2b,2b)U_0(2b)P^b_-f,\  U_0(2b)\mathcal{U}(t-2b,0)P^b_-f\in D^b_+,\quad  f\in{\B}.\]
Therefore,
\[Z^b(t,0)=P^b_-M(t,t-2b)\chi \mathcal{U}(t-2b,2b)\chi M(2b,0)P^b_-.\]
Since $t-4b>T_{4b}$, Proposition 2 implies that
\[\chi \mathcal{U}(t-2b,2b)\chi h\in{\CI},\quad   h\in{\B}.\]
Thus,  we conclude that $\chi \mathcal{U}(t-2b,2b)\chi$ is a compact operator in ${\B}$ and the proof is complete.\end{proof}

As for the wave equation with time dependent potential (see [18], chapter V) we can prove that for a  non-trapping metric we have 
\[\sigma(Z^b(T,0))=\sigma(Z^\rho(T,0)),\quad   b\geq\rho\]
and we omit the details.
Combining this with the assumption (H2), we get
\begin{prop}
 For all $t,s\in{\R}$ with $t-s\geq0$ and $b\geq\rho$ there exist $C_b,\delta_b>0$ independent on $t,s$ such that
\[\Vert Z^b(t,s)\Vert_{\mathcal{L}({\B})}\leq C_be^{\delta_b(t-s)}.\]
\end{prop}
\begin{proof}
Let $r(Z^b(T,0))$ be the spectral radius of $Z^b(T,0)$.
We know that\\
 $r(Z^b(T,0))<1$. Thus  there exists $\delta>0$ such that 
\[\lim_{m\to+\infty}\Vert Z^b(mT,0)\Vert^\frac{1}{m}<1-\delta\]
and for $m\geq m_0$ we have 
\[\Vert Z^b(mT,0)\Vert\leq(1-\frac{\delta}{2})^m=e^{-\delta_bmT}\]
with $\delta_b>0$. Now assume that $t-s\geq (m_0+2)T$ and choose $k,l\in\mathbb{N}$ so that
\[kT\leq t\leq(k+1)T,\quad lT\leq s\leq (l+1)T.\]
Then 
\[\Vert Z^b(t,s)\Vert=\Vert Z^b(t,kT)Z^b(kT,(l+1)T)Z^b((l+1)T,s))\Vert\]
and $(k-(l+1))T\geq m_0T$. Thus we obtain
\[\Vert Z^b(t,s)\Vert\leq C'_be^{-\delta_b(k-(l+1))T}\leq C'_be^{-\delta_b(t-s)}e^{2\delta_bT}.\]
For $t-s\leq(m_0+2)T$ we have the estimate
\[\Vert Z^b(t,s)\Vert\leq C''_be^{\alpha(m_0+2)T}\leq C_b''e^{(\alpha+\delta_b)(m_0+2)T}e^{-\delta_b(t-s)}.\]
This completes the proof.\end{proof}

Now we are able  to show the main result of this section.
\begin{Thm}
 Let $\phi\in{\CI}$ be a cut-off function and let  $f\in {\B}$ be such that\\
supp$f\subset\{x\in{\R}^n: \vert x\vert<R\}$. Assume the condition (H2) is fulfilled and let $n\geq3$ be odd. Then there exist $C,\delta>0$ independent of $f$ and $t$ so that for $t\geq0$ we have
\[\Vert \phi \mathcal{U}(t,0)f\Vert_{{\B}}\leq C(\rho,R,\phi,n)e^{-\delta t}\Vert f\Vert_{{\B}}.\]
\end{Thm}
\begin{proof}
Set $b=max(\rho+1,R+1,R_1)$, where
 supp$(\phi)\subset\{x\in{\R}^n:\vert x\vert<R_1\}$. Applying Huygens principle, we know that  $(U_0(t)f)(x)=0$ for $\vert x\vert<t-b$. We get $f\in D^b_+\subset(D^b_-)^\bot$
and $f=P_-^b(f)$. Also, for all $f\in {\B}$, $(Id-P_+^b)(f)\in D^b_+$and we obtain $(Id-P_+^b)(f)\vert_{ \vert x\vert<b}=0$. Thus
\[\phi(Id-P_+^b)(f)=0,\quad  f\in{\B}\]
and
\[\phi \mathcal{U}(t,0)f=\phi(Id-P_+^b)\mathcal{U}(t,0)P_-^b(f)+\phi P_+^b\mathcal{U}(t,0)P_-^b(f)=\phi P_+^b\mathcal{U}(t,0)P_-^b(f)=\phi Z^b(t,0)f.\]
Consequently,
\[\Vert \phi \mathcal{U}(t,0)f\Vert_{{\B}}=\Vert\phi Z^b(t,0)f\Vert_{{\B}}\leq C(\phi,n)\Vert Z^b(t,0)\Vert_{\mathcal{L}({\B})}\Vert f\Vert_{\B}.\]
Proposition 4 implies that for all $t\geq0$ we have
\[\Vert\phi \mathcal{U}(t,0)f\Vert\leq C(\phi,n,\rho,R,T)e^{-\delta_bt}\Vert f\Vert_{\B}.\]
\end{proof}

\section{ $L^2$ integrability of the local energy}
\renewcommand{\theequation}{\arabic{section}.\arabic{equation}}
\setcounter{equation}{0}

First, we will recall two useful results. The first  result says that for functions  $u\in\dot{H}^s$ with compact support and all  $\gamma\leq s<\frac{n}{2}$,  
$ \Vert u\Vert_{H^\gamma}\leq C\Vert u\Vert_{\dot{H}^s}$. The second  one established by Smith et Sogge in [25] concerns the $L^2$ integrability in time of the local energy for solutions of the free wave equation.

\begin{Thm}$[25]$
 Let $\gamma\leq \frac{n-1}{2}$ and let  $\phi\in{\CI}$. Then
\[\int_{\R}\Vert \phi e^{\pm it\Lambda}f\Vert_{H^\gamma({\R}^n)}^2\textrm{d}t\leq C(\phi,n,\gamma)\Vert f\Vert_{\dot{H}^\gamma({\R}^n)}^2.\]
\end{Thm}
This theorem  implies the following
\begin{cor}
 Let $\gamma\leq\frac{n-1}{2}$ and $\phi\in{\CI}$. Then
\[\int_{\R}\Vert \phi U_0(t)f\Vert^2_{\mathcal{H}_\gamma}\textrm{d}t\leq C(\phi,n,\gamma)\Vert f\Vert_{\dot{\mathcal{H}}_\gamma}^2.\]
\end{cor}

Combining this estimate and the link between the free wave equation and problem (1.1), we deduce our main result in this section.
\begin{Thm}
 Let $\phi\in{\CI}$, and let $n\geq 3$ be odd. Then 
\[\int_0^\infty\Vert \phi \mathcal{U}(t,0)f\Vert_{{\B}}^2\textrm{d}t\leq C(T,\phi,n,\rho)\Vert f\Vert_{{\B}}^2.\]
\end{Thm}

\begin{proof}Let $\phi\in{\CI}$, $f\in {\B}$ and $\chi\in{\CI}$ be such that\\
 $\chi=1$ for $ \vert x\vert\leq \rho +\half$ and $supp\chi\subset\{\ x\ :\ \vert x\vert\leq\rho+1\}$, $0\leq\chi\leq1$. Notice that
\[\phi \mathcal{U}(t,0)f=\phi \mathcal{U}(t,0)\chi f+ \phi \mathcal{U}(t,0)(1-\chi)f.\]
Theorem 2 implies
\[\int_0^{+\infty} \Vert\phi \mathcal{U}(t,0)\chi f\Vert^2_{\B}\textrm{d}t\leq C^2(\phi,n,\rho) (\int_0^{+\infty}e^{-2\delta_bt}\textrm{d}t)\Vert f\Vert_{\B}^2\leq C_1(\phi,n,\rho)\Vert f\Vert_{\B}.\]
It remains to show that
\[\Vert \phi \mathcal{U}(t,0)(1-\chi)f\Vert_{L^2({\R}_t^+,{\B})}\leq C(\phi,n,\rho)\Vert f\Vert^2_{\B}.\]
  Let  $u,u_0,u_1,u_2$ be defined by
\[\left(u_0,\partial_t(u_0)\right)(t)=U_0(t)f,\quad \left(u,\partial_t(u)\right)(t)=\mathcal{U}(t,0)(1-\chi)f,\quad u_1=(1-\chi)u_0\quad\textrm{and }u_2=u-u_1.\]
We have
\[\partial^2_t(u_1)-\Div_x(a(t,x)\nabla_x (u_1))=(\partial^2_t-\Delta_x)u_1.\]
Also 
\[(\partial_t^2-\Delta_x)u_1=(1-\chi)(\partial_t^2-\Delta_x)(u_0)+[\Delta_x,\chi](u_0)=[\Delta_x,\chi](u_0).\]
Thus $u_2$ is the solution of the following Cauchy problem

\[  \left\{\begin{array}{c}
\partial_t^2(u_2)-\Div_x(a\nabla_x u_2)=\partial_t^2u-\Div_x(a\nabla_x u) -(\partial_t^2u_1-\Div_x(a\nabla_x u_1))=-[\Delta_x,\chi]u_0,\\
(u_2,\partial_t(u_2))(0)=0.\end{array}\right.\]
Therefore,  we can write
\[\left( u_2, \partial_t(u_2)\right)(t)=-\int_0^t\mathcal{U}(t,s)\left(0,[\Delta_x,\chi]u_0(s)\right)\textrm{d}s.\]
On the other hand, supp$\left(0,[\Delta_x,\chi]u_0(s)\right)\subset$supp$\chi\subset\{\vert x\vert<\rho+1\}$ and applying
 Theorem 2,  with $R=\rho +1$  we find
\[\left\Vert \phi \mathcal{U}(t,s)\left(0,[\Delta_x,\chi]u_0(s)\right)\right\Vert_{\B}\leq C(\rho,n)e^{-\delta_b(t-s)}\left\Vert \left(0,[\Delta_x,\chi]u_0(s)\right)\right\Vert_{\B}.\]
Choosing a cut-off function $\beta\in{\CI}$  equal to 1 on supp$\chi$, we get
\[\left\Vert \left(0,[\Delta_x,\chi]u_0(s)\right)\right\Vert_{\B}\leq C(\rho,n) \Vert \beta u_0(s)\Vert_{\dot{H}^1}\leq C(\rho,n)\Vert\beta U_0(s)f\Vert_{\B}.\]
Therefore,
\[\begin{array}{lll}\displaystyle\left\Vert \phi\left(u_2,\partial_t(u_2)\right)(t)\right\Vert_{\B}&\leq& C(\rho,n,\phi)\int_0^te^{-\delta_b(t-s)}\Vert\beta U_0(s)f\Vert_{\B}\textrm{d}s,\\ \ &\leq&C(\rho,n,\phi)(e^{-\delta_bt}1_{[0,+\infty[})*(\Vert\beta U_0(t)f\Vert_{\B}1_{[0,+\infty[}).\end{array}\]
An application of Young inequality  yields
\[\left(\int_{{\R}^+}\Vert\phi \left(u_2,\\ \partial_t(u_2)\right)(t)\Vert_{\B}^2\textrm{d}t\right)^{\half}\leq
C(\rho,n,\phi)\left(\int_0^{+\infty}e^{-\delta_bt}\textrm{d}t\right)\left(\int_{{\R}^+}\Vert\beta U_0(t)f\Vert_{\B}^2\textrm{d}t\right)^\half.\]
Since $1\leq\frac{n-1}{2}$, Corollary 1 shows that 
\[\left\Vert\phi \left( u_2,\partial_t(u_2)\right)(t)\right\Vert_{L^2({\R}^+,{\B})}\leq C(\phi,\rho,n)\Vert f\Vert_{\B}.\]
Consequently,
\begin{equation} \label{eq:2}\Vert \phi \mathcal{U}(t,0)(1-\chi)f\Vert_{L^2({\R}^+,{\B})}
\leq C(\phi,\rho,n) \Vert f\Vert_{\B}.\end{equation}	
In conclusion, we obtain
\[\begin{array}{lll}\Vert \phi \mathcal{U}(t,0) f\Vert_{L^2({\R}^+,{\B})}&\leq& \Vert \phi \mathcal{U}(t,0)\chi f\Vert_{L^2({\R}^+,{\B})} \\
\ &\ &\ + \Vert \phi \mathcal{U}(t,0)(1-\chi) f\Vert_{L^2({\R}^+,{\B})},\\ \ &\leq& C(\rho,\phi,n)\Vert f\Vert_{\B}.\\ \end{array}\]
\end{proof}

\vspace{1cm}

 Consider the solution $u$ of (1.1) with $s=0$ and 
let $\chi\in {\CI}$ be a cut-off function satisfying supp$\chi\subset B_{\rho+1}$, $\chi(x)=1$ $  x\in B_{\rho +\half}$ and $0\leq \chi\leq1$. We will establish Strichartz estimates for  $(1-\chi)u$ and $\chi u$. Notice that
\[\partial_t^2((1-\chi)u)-\Delta((1-\chi)u)=(1-\chi)(\partial_t^2-\Div_x(a\nabla_x u))+[\Delta_x,\chi]u= [\Delta_x,\chi]u.\]
Therefore, $(1-\chi)u$ is the solution of the free wave equation with right hand part \newline
$[\Delta_x,\chi]u\in L^2_t({\R}^+,L^2_x({\R}^n))$ and initial data $(1-\chi)f$. We will apply the result of
[19] to obtain Strichartz estimate for $(1-\chi)u$.
On the other hand, $\chi u$ satisfies
\[\partial^2_t(\chi u)-\Div_x(a\nabla_x (\chi u))=\chi(\partial_t^2(u)-\Div_x(a\nabla_x u)) -[\Div_x(a\nabla_x),\chi]u=-[\Div_x(a\nabla_x),\chi]u,\]
and $(\chi u,\partial_t(\chi u))(0)=\chi f$. To deal with this term  we will exploit local Strichartz estimates. 
\section{ Strichartz estimate for $(1-\chi)u$}
\renewcommand{\theequation}{\arabic{section}.\arabic{equation}}
\setcounter{equation}{0}
In this section our purpose is to prove the following.
\begin{prop} 
Let $2\leq p, q<\infty$ satisfy inequalities  (1.3) with $p>2$. Then for all  $t>0$ we have
\[\Vert (1-\chi)u\Vert_{L^p({\R}^+,L^q({R}^n))}+ \Vert(1-\chi) u(t)\Vert_{\dot{H}^1({\R}^n)}+\Vert \partial_t(1-\chi)(u)(t)\Vert_{L^2({\R}^n)}\leq C(\Vert f_1\Vert_{\dot{H}^1({\R}^n)}+\Vert f_2\Vert_{L^2({\R}^n)}).\]

\end{prop}
To prove Proposition 5 we need two lemmas.
\begin{lem}
 Let $X$ and $Y$ be Banach spaces, and for all $s,t\in{\R}^+$ let $K(s,t):X\longrightarrow Y$ be an operator-valued kernel from $X$ to $Y$. Suppose that
\[\left\Vert\int_0^{t_0}K(s,t)g(s)\textrm{d}s\right\Vert_{L^l([t_0,+\infty[,Y)}\leq A\Vert g\Vert_{L^r({\R}^+,X)}\]
for some $A>0$, $1\leq r<l\leq+\infty$,  all $t_0\in{\R}^+$ and $g\in L^r({\R}^+,X)$. Then we have

\[\left\Vert\int_0^tK(s,t)g(s)\textrm{d}s\right\Vert_{L^l({\R}^+,Y)}\leq AC_{r,l}\Vert g\Vert_{L^r({\R}^+,X)}\]
where $C_{r,l}>0$ depends only on $r,l$.
 \end{lem}
We refer to  [11] for the proof of Lemma 1 which is called the Christ-Kiselev lemma (see also the original paper [5]). Next, consider 
\[K(s,t)=\frac{\sin((t-s)\Lambda)}{\Lambda}\psi,\ \ X=L^2({\R}^n),\ \ Y=L^q({\R}^n),\ \ l=p\ \ \textrm{and }r=2.\]
Applying the Christ-Kiselev lemma and Strichartz estimates for the free wave equation (see [14]) we get the following .
\begin{lem}
 Let $p$ and $q$ satisfy  (1.3) with $p>2$ and $\gamma=1$. Then for all $\psi\in{\CI}$ we have
\[\left\Vert\int_0^t\frac{\sin((t-s)\Lambda)}{\Lambda}\psi h(s,.)\textrm{d}s\right\Vert_{L^p({\R},L^q({\R}^n))}\leq C(p,q,n,\psi)\Vert h\Vert_{L^2({\R}^+,L^2({\R}^n))}\]
and
\[\left\Vert\int_0^t\frac{\sin( (t-s)\Lambda)}{\Lambda}\psi h(s,.)\textrm{d}s\right\Vert_{L^2({\R}^n)}\leq C(p,q,n,\psi)\Vert h\Vert_{L^2({\R}^+,L^2({\R}^n))}.\]
\end{lem}
We refer to [19] for the proof of  Lemme 2.\\
\textit{Proof of  Proposition 5:}  As we mentioned above, $(1-\chi)u$ is the solution of the free wave equation with right hand term
 $[\Delta_x,\chi]u\in L^2_t({\R}^+,L^2_x({\R}^n))$ and initial data\\
  $((1-\chi)u,(1-\chi)u_t)(0)=((1-\chi)f_1,(1-\chi)f_2)$. Thus
 \[(1-\chi)u(t)=\cos(t\Lambda)(1-\chi)f_1+\frac{\sin(t\Lambda)}{\Lambda}(1-\chi)f_2+\int_0^t\frac{\sin((t-s)\Lambda)}{\Lambda}[\Delta_x,\chi]u(s)\textrm{d}s.\]
  Exploiting the global Strichatz estimates for the free wave equation, Lemma 2 and Theorem 4  we obtain immediately Proposition 5. $\quad\quad\quad\quad\quad\quad\quad\quad\quad\quad\quad\quad\quad\quad\quad\quad\quad\quad\quad\quad\quad\quad\quad\quad\quad\quad\quad\quad\square$
\vspace{1cm}

\section{ Strichartz  estimates for $\chi u$}
\renewcommand{\theequation}{\arabic{section}.\arabic{equation}}
\setcounter{equation}{0}

In this  section we  admit the following local Strichartz estimates which will be established in Section 7.
\begin{Thm}
 Let  $\psi\in{\CI}$. Then there exists $\delta>0$ such that for  $2\leq p,q<+\infty$, $s\in[0,T]$ and $\gamma>0$ satisfying
\begin{equation}\label{eq:1}\frac{n(q-2)}{2q}-\gamma=\frac{1}{p}\leq\frac{(n-1)(q-2)}{4q},\end{equation}
we have
 \[\int_s^{s+\delta} \Vert\psi (\mathcal{U}(t,s)f)_1\Vert_{L^q({\R}^n)}^p\textrm{d}t\leq C(T,\psi,\delta ,p,q,n)\Vert f\Vert_{\dot{\mathcal{H}}_\gamma}^p,\]
where $\delta$ and $C>0$ are independent on $s$ and $f$.
\end{Thm}
 Notice that for $h=(h_1,h_2)$ we consider that $(h)_1=h_1$. We admit this result to complete the estimates of $\chi u$. We  apply  an argument similar to that used by  N.Burq in [1]. First, as in [1], consider a cut-off functions $\phi\in\mathcal{C}^\infty_0({\R})$ such that supp$\phi\subset]0,\delta[$, $0\leq\phi\leq1$ and \newline$\phi(t)=1,\quad  t\in[\frac{\delta}{4},\frac{3\delta}{4}]$.
Set \[\phi_\nu(t)=\phi(t-\frac{\nu\delta}{2}),\ \ \nu\in\mathbb{Z}.\]
Clearly we have   supp$\phi_\nu\bigcap$supp$\phi_{\nu+2}=\varnothing$ and 
\[1\leq \sum_{\nu=-\infty}^{+\infty}\phi_\nu(t)\leq2,\quad  t\in{\R}.\]
We will  apply Theorem 5 to establish the following.
\begin{lem}
 Let $\tau>0$ and let $2\leq p,q<+\infty$ satisfy the conditions (5.1) with $\gamma=1$. Then we have
\[\Vert (\chi\mathcal{U}(t,\tau)f)_1\Vert_{L^p([\tau,\tau+\delta],L^q({\R}^n))}\leq C\Vert f\Vert_{\B}\]
with $C>0$ independent of $\tau$ and $f$.
\end{lem}
\begin{proof}
Take $k\in\mathbb{N}$ and $s$ such that $kT\leq \tau<(k+1)T$ and $s=\tau-kT\in[0,T]$. We get
\[\mathcal{U}(t,\tau)f=\mathcal{U}(t-kT,s)f.\]
Then
\[\int_\tau^{\tau+\delta}\Vert(\chi\mathcal{U}(t,\tau)f)_1\Vert_{L^q({\R}^n)}^p\textrm{d}t=\int_s^{s+\delta}\Vert(\chi\mathcal{U}(t',s)f)_1\Vert_{L^q({\R}^n)}^p\textrm{d}t'\]
and Theorem 5 implies
\[\int_s^{s+\delta}\Vert(\chi\mathcal{U}(t',s)f)_1\Vert_{L^q({\R}^n)}^p\textrm{d}t'\leq C\Vert f\Vert^p_{\B},\]
with $C>0$ independent on $s$. 
\end{proof}

\begin{prop}
Let $2\leq p,q\leq\infty$ satisfy condition (5.1) with $\gamma=1$  and let\\
 $\psi\in{\CI}$. Then for all $g\in L^1([0,T_1],L^2({\R}^n))$ we have 
\[\left\Vert\int_\tau^t\psi \left(\mathcal{U}(t,s)\left(\begin{array}{l}0\\ g(s)\\ \end{array}\right)\right)_1\textrm{d}s\right\Vert_{L^p([\tau,\delta+\tau],L^q({\R}^n))}\leq C\Vert g\Vert_{L^1([\tau,\delta+\tau],L^2({\R}^n))}\]
with $C>0$ independent of $g$ and $\tau$.\end{prop}
\begin{proof}
We have
\[\begin{array}{c}\left\Vert\int_\tau^t\psi \left(\mathcal{U}(t,s)\left(\begin{array}{l}0\\ g(s)\\ \end{array}\right)\right)_1\textrm{d}s\right\Vert_{L^p([\tau,\delta+\tau],L^q({\R}^n))}\\
\leq\int_\tau^{\tau+\delta}\left\Vert1_{\{t>s\}}\psi \left(\mathcal{U}(t,s)\left(\begin{array}{l}0\\ g(s)\\ \end{array}\right)\right)_1\right\Vert_{L_t^p([\tau,\delta+\tau],L^q)}\textrm{d}s.\end{array}\]
Thus Lemma 3 implies
\[\left\Vert\int_\tau^t\psi \left(\mathcal{U}(t,s)\left(\begin{array}{l}0\\ g(s)\\ \end{array}\right)\right)_1\textrm{d}s\right\Vert_{L^p([\tau,\delta+\tau],L^q({\R}^n))}\leq C\Vert g\Vert_{L^1([\tau,\delta+\tau],L^2({\R}^n)}\]
with $C>0$ independent of $g$ and $\tau$.
\end{proof}

Consider $v_\nu=\phi_\nu\chi u$, $I_\nu=]\frac{\nu\delta}{2},\frac{\nu\delta}{2}+\delta[\bigcap{\R}^+$ and
 \[u_\nu=[\partial^2_t,\phi_\nu]\chi u-[\Div_x(a\nabla_x\cdot),\chi]\phi_\nu u.\]
We see that $v_\nu$ is the solution of the problem
\begin{equation} \label{eq:1}  \left\{\begin{array}{c}
\partial^2_t(v_\nu)-\Div_{x}(a(t,x)\nabla_{x}v_\nu)=u_\nu,\\
(v_\nu,\partial_t(v_\nu))(0)=g_\nu,\end{array}\right.\end{equation}
with $g_{-1}=\chi f$ and $g_\nu=0$ for $\nu\neq-1$. Let $\beta\in{\CI}$ be such that $\beta=1$ on $supp \chi$. We deduce that
\[v_\nu(t)=\beta v_\nu(t)=(\beta \mathcal{U}(t,0)g_\nu )_1+\int_0^t\left(\beta \mathcal{U}(t,s)\left(\begin{array}{l} 0\\ u_\nu(s)\\ \end{array}\right)\right)_1\textrm{d}s.\]
Now let $p,\ q,$ satisfy condition (5.1) with $\gamma=1$.  Applying  Theorem 5 and Proposition 6,  we get
\[\Vert v_\nu\Vert_{L^p({\R}^+,L^q({\R}^n))}=\Vert v_\nu\Vert_{L^p(I_\nu,L^q({\R}^n))}\leq C(\Vert g_\nu\Vert_{{\B}} +\Vert u_\nu\Vert_{L^1({\R}^+,L^2({\R}^n))}).\]
In the same way
\[\Vert (v_\nu,\partial_t(v_\nu))\Vert_{\mathcal{C}({\R}^+,{\B})}\leq C(\Vert g_\nu\Vert_{{\B}} +\Vert u_\nu\Vert_{L^1({\R}^+,L^2({\R}^n))}).\]
Using   Minkowski's inequality, we find
\[\left\Vert\left(\sum_{\nu=-N}^{N} \vert v_\nu\vert\right)\right\Vert_{L^q({\R}^n)}\leq\sum_{\nu=-N}^{N} \Vert v_\nu\Vert_{L^q({\R}^n)}\leq\sum_{\nu=-\infty}^{+\infty} \Vert v_\nu\Vert_{L^q({\R}^n)},\   N\in\mathbb{N}.\]
Thus, taking the limit $N\rightarrow+\infty$, we get
\[\left\Vert\left(\sum_{\nu=-\infty}^{+\infty} \vert v_\nu\vert\right)\right\Vert_{L^q({\R}^n)}=\lim_{N\to\infty}\left\Vert\left(\sum_{\nu=-N}^{N} \vert v_\nu\vert\right)\right\Vert_{L^q({\R}^n)}\leq\sum_{\nu=-\infty}^{+\infty} \Vert v_\nu\Vert_{L^q({\R}^n)}\]
and we obtain
\[\Vert \chi u\Vert_{L^q({\R}^n)}\leq\left\Vert\left(\sum_{\nu=-\infty}^{+\infty} \vert v_\nu\vert\right)\right\Vert_{L^q({\R}^n)}\leq\sum_{\nu=-\infty}^{+\infty} \Vert v_\nu\Vert_{L^q({\R}^n)}.\]
On the other hand, $p\geq 1$ and we have
\[\sum_{\nu=-\infty}^{+\infty}\Vert v_\nu(t,.)\Vert_{L^q({\R}^n)}=\Vert v_{l-1}(t,.)\Vert_{L^q({\R}^n)}+\Vert v_{l}(t,.)\Vert_{L^q({\R}^n)}+\Vert v_{l+1}(t,.)\Vert_{L^q({\R}^n)},\ \   t\in I_l\]
and we deduce that for all $\nu\in\mathbb{Z}$
\[\begin{array}{lll}\int_{I_\nu}\Vert\chi u(t,.)\Vert_{L^q({\R}^n)}^p\textrm{d}t\leq C&(&\int_{{\R}^+}((\Vert v_{\nu-1}(t,.)\Vert_{L^q({\R}^n)})^p+(\Vert v_{\nu}(t,.)\Vert_{L^q({\R}^n)})^p+\\

\ &\ &+(\Vert v_{\nu+1}(t,.)\Vert_{L^q({\R}^n)})^p)\textrm{d}t).\end{array}\]
Consequently, 
\[\int_{{\R}^+}\Vert\chi u(t,.)\Vert_{L^q({\R}^n)}^p\leq \sum_{\nu=-\infty}^{+\infty}\int_{I_\nu}\Vert u(t,.)\Vert_{L^q({\R}^n)}^p\leq C\sum_{\nu=-\infty}^{+\infty}\Vert v_\nu\Vert_{L^p({\R}^+,L^q({\R}^n))}^p,\]
and
\[\sum_{\nu=-\infty}^{+\infty} \Vert v_\nu\Vert_{L^p({\R}^+,L^q({\R}^n))}^p\leq C\sum_{\nu=-\infty}^{+\infty}(\Vert u_\nu\Vert_{L^1({\R}^+,L^2({\R}^n))}+\Vert g_\nu\Vert_{\B})^p.\]
It is clear there exists $C>0$ independent of $\nu$ such that
\[
\Vert u_\nu\Vert_{L^1({\R}^+,L^2({\R}^n))}\leq C\Vert u_\nu\Vert_{L^2({\R}^+,L^2({\R}^n))}\]
and we find
\[\sum_{\nu=-\infty}^{+\infty} \Vert v_\nu\Vert_{L^p({\R}^+,L^q({\R}^n))}^p\leq C
\sum_{\nu=-\infty}^{+\infty}((\Vert u_\nu\Vert_{L^2({\R}^+,L^2({\R}^n))}^2+\Vert g_\nu\Vert_{\B}^2)^{\frac{p}{2}}).\]

\begin{lem}
Let $r\geq1$. Then for all complex valued sequences $(a_k)_{k\in\mathbb{Z}}$ we have
\[\sum_{k=-\infty}^{+\infty}\vert a_k\vert^r\leq\left(\sum_{k=-\infty}^{+\infty}\vert a_k\vert\right)^r.\]
\end{lem}
\begin{proof}
Let $\sum_{-\infty}^{+\infty}\vert a_k\vert=1$. Then \[\vert a_k\vert^r\leq \vert a_k\vert,\quad  k\in\mathbb{Z}\] and
\[\sum_{-N}^{+N}\vert a_k\vert^r\leq\sum_{-N}^{+N}\vert a_k\vert,\]
which implies 
\[\sum_{-\infty}^{+\infty}\vert a_k\vert^r\leq 1.\]
Now set $\alpha=\sum_{-\infty}^{+\infty}\vert a_k\vert$ and consider $b_k=\frac{a_k}{\alpha}$. We obtain
\[\frac{\sum_{-\infty}^{+\infty}\vert a_k\vert^r}{\alpha^r}\leq1\]
and 
\[\sum_{-\infty}^{+\infty}\vert a_k\vert^r\leq(\sum_{-\infty}^{+\infty}\vert a_k\vert)^r.\]
\end{proof}

Since $\frac{p}{2}\geq1$ we deduce from Lemma 4 the estimate
\[\sum_{\nu=-\infty}^{+\infty}(\Vert u_\nu\Vert_{L^2({\R}^+,L^2({\R}^n))}^2+\Vert g_\nu\Vert_{\B}^2)^{\frac{p}{2}}\leq\left(\sum_{\nu=-\infty}^{+\infty}(\Vert u_\nu\Vert_{L^2({\R}^+,L^2({\R}^n))}^2+ +\Vert g_\nu\Vert_{\B}^2)\right)^{\frac{p}{2}}.\]
Let $\beta\in{\CI}$ be a cut-off function such that $\beta=1$ on supp$\chi$. It follows from Theorem 4 that
\[\sum_{\nu=-\infty}^{+\infty}\Vert u_\nu\Vert^2_{L^2({\R}^+,L^2({\R}^n))}\leq C\int_0^{+\infty}\Vert\beta(u,u_t)(t)\Vert_{\B}^2\textrm{d}t\leq C(\rho,n,T)(\Vert f\Vert_{\B})^2\]
and
\[\sum_{\nu=-\infty}^{+\infty}\Vert g_\nu\Vert^2=\Vert\chi f\Vert^2.\]
This result shows  that
\begin{equation} \label{eq:2}\Vert \chi u\Vert_{L^p({\R}^+,L^q)}\leq\left(\sum_{\nu=-\infty}^{+\infty}\Vert u_\nu\Vert_{L^2({\R}^+,L^2({\R}^n))}^2+\Vert g_\nu\Vert_{\B}^2\right)^{\frac{1}{2}}\leq C\Vert f\Vert_{\B}.\end{equation}
\begin{prop}
Let $n\geq3$ be odd and (H1), (H2) be fulfilled. Then, for all $f\in{\B}$ and for all $t>0$ we have
\begin{equation} \label{eq:2}\Vert \chi\U(t,0)f\Vert_{\B}\leq C(\chi,n,\rho,T)\Vert f\Vert_{\B}\end{equation}
with $C(\chi,n,\rho,T)$ independent on $t$ and $f$.
\end{prop}
\begin{proof}
Let $f\in{\B}$ and let $t>0$.
Notice that
\[\chi \mathcal{U}(t,0)f=\chi \mathcal{U}(t,0)\chi f+ \chi \mathcal{U}(t,0)(1-\chi)f.\]
Theorem 2 implies
\[ \Vert\chi \mathcal{U}(t,0)\chi f\Vert_{\B}\leq C(\chi,n,\rho) e^{-\delta t}\Vert f\Vert_{\B}\leq C'\Vert f\Vert_{\B},\]
with $C'>0$ independent on $t$ and $f$. Moreover, as in (3.1) we have

\[\Vert \chi \mathcal{U}(t,0)(1-\chi)f\Vert_{\B}\leq C\Vert f\Vert_{\B},\]
with $C>0$ independent on $t$ and $f$. 
\end{proof}
In conclusion, estimates (5.3) and (5.4) imply that if $2\leq p,q\leq\infty$ satisfy condition (5.1) with $\gamma=1$, for all $t>0$ we have 
\[\Vert\chi u\Vert_{L^p({\R}^+,L^q({R}^n))}+ \Vert\chi u(t)\Vert_{\dot{H}^1({\R}^n)}+\Vert \partial_t(\chi u)(t)\Vert_{L^2({\R}^n)}\leq C(\rho,T,n,p,q)(\Vert f_1\Vert_{\dot{H}^1({\R}^n)}+\Vert f_2\Vert_{L^2({\R}^n)}).\]

\section{Proof of Theorem 1} 
\renewcommand{\theequation}{\arabic{section}.\arabic{equation}}
\setcounter{equation}{0}
In this  section we will apply  the results of previous sections to establish  Theorem 1. First, Proposition 6 implies that for $2\leq p,q<+\infty$ satisfying
\[p>2,\quad \frac{1}{p}=n\left(\half-\frac{1}{q}\right)-1,\quad\textrm{et}\quad \frac{1}{p}\leq \frac{n-1}{2}\left(\half-\frac{1}{q}\right)\]
 and for all $t>0$ we have
\[\Vert(1-\chi)u\Vert_{L^p({\R}^+,L^q({\R}^n))}+\Vert(1-\chi)(u,u_t)(t)\Vert_{\B}\leq C(\rho,n,p,q,T)\Vert f\Vert_{\B}.\]
Also, for $2\leq p,q<+\infty$ satisfying
\[\frac{1}{p}= \frac{n(q-2)}{2q}-1=n\left(\half-\frac{1}{q}\right)-1,\quad\textrm{and}\quad \frac{1}{p}\leq\left(\frac{n-1}{2}\right)\left(\frac{q-2}{2q}\right)= \frac{n-1}{2}\left(\half-\frac{1}{q}\right)\]
 and for all $t>0$ we deduce from the results of Section 5 the estimates
\[\Vert\chi u\Vert_{L^p({\R}^+,L^q({\R}^n))}+\Vert\chi(u,u_t)(t)\Vert_{\B}\leq C(\rho,n,p,q,T)\Vert f\Vert_{\B}.\]
Consequently, for $2\leq p,q<+\infty$ such that
\begin{equation}\label{eq:1}p>2,\quad \frac{1}{p} =\frac{n(q-2)}{2q}-1,\quad\textrm{and}\quad \frac{1}{p}\leq\left(\frac{n-1}{2}\right)\left(\frac{q-2}{2q}\right),\end{equation}
 and all $t>0$ we get
\[\Vert u\Vert_{L^p({\R}^+,L^q({\R}^n))}+\Vert(u,u_t)(t)\Vert_{\B}\leq C(\rho,n,p,q,T)\Vert f\Vert_{\B}.\]
\textit{Proof of Theorem 1:}
We know that $p$ and $q$  satisfy (6.1)   if
\[0<\frac{1}{p}=\frac{n(q-2)}{2q}-1<\half\]
and
\[\frac{1}{p}\leq\left(\frac{n-1}{2}\right)\left(\frac{q-2}{2q}\right).\]
Consequently, $q$  satisfies
\begin{equation}\label{eq:1}\left\{\begin{array}{lll}
(n-3)q\leq2(n+1),\\
(n-3)q< 2n,\\
q>\frac{2n}{n-2},\\
\end{array}\right.\end{equation}
and $\frac{1}{p}=\frac{n(q-2)}{2q}-1$.
For $n=3$ the conditions (6.1) are fulfilled for $q>\frac{2n}{n-2}=6$,  and $\frac{1}{p}=\frac{n(q-2)}{2q}-1$. 
For $n\geq5$ odd , we find that $q$ satisfies (6.2)  if 
\[\frac{2n}{n-2}<q<\frac{2n}{n-3}.\]
Therefore, if  $\frac{2n}{n-2}<q<\frac{2n}{n-3}$ and $\frac{1}{p}=\frac{n(q-2)}{2q}-1$ we conclude that $p$ and $q$ satisfy  conditions (6.1).$\square$

\section{ Local Strichartz estimates for  solutions of the disturbed wave equation }
\renewcommand{\theequation}{\arabic{section}.\arabic{equation}}
\setcounter{equation}{0}
 The purpose of this section is to prove  Theorem 5. To establish this result we will show that we can approximate the solutions of the problem (1.1) by Fourier integral operators. Then we will apply the Strichartz estimates  of Kapitanski (see [12]) for Fourier integral operators on Besov space to get the local Strichartz estimates.

\subsection{Approximation of the propagator $\mathcal{U}(t,\tau)$ by Fourier integral operators}

Consider the operators $U(t,s)$ and $V(t,s)$ defined by
\[ U(t,s)f=\left(\mathcal{U}(t,s)\left(\begin{array}{l}f\\ 0\\ \end{array}\right)\right)_1,\quad f\in\dot{H}^\gamma({\R}^n),\]
\[ V(t,s)g=\left(\mathcal{U}(t,s)\left(\begin{array}{l}0\\ g\\ \end{array}\right)\right)_1,\quad g\in \dot{H}^{\gamma-1}({\R}^n).\]
We denote by $B^m$ the space 
\[B^m=\bigcap_{s\in{\R}}\mathcal{L}(H^s({\R}^n),H^{s-m}({\R}^n)).\]
Let $B(t)\in B^k$ for all $t\in[S_1,S_2]$. We say that the operator $B(t)$ depends on $t$ in an admissible fashion if $\partial_t^jB(t)\in B^{k+j}$, $j=1,2,...$ and
\[ \Vert\partial_t^jB(t)\Vert_{\mathcal{L}(H^s,H^{s-k-j})}\leq C_j,\quad  s\in{\R}\]
with $C_j>0$ independent of $t$, $S_1$ and $S_2$.
In this subsection we will  establish the following
\begin{Thm}
Let $\psi\in\mathcal{C}^\infty_0(\vert x\vert\leq R_1)$, $R_1>\rho$. Then there exists $\delta>0$ such that for $s,t\in[0,T]$ with $\vert s-t\vert<\delta$ and every integer $N\geq1$ we have the representation

\[\psi U(t,s)=\sum_{j=1}^M(\tilde{I}_j^+(t,s)+\tilde{I}_j^-(t,s))+ R_N(t,s),\]
where $\tilde{I}_j^\pm(t,s)$ are Fourier integral operators with kernels
\[\tilde{I}^\pm_k(t,s,x,y)=\int \tilde{b}_k^\pm(t,s,y,x,\xi)e^{-i\phi^\pm_k(t,s,y,\xi)+ix.\xi}\textrm{d}\xi\]
and $R_N(t,s)\in B^{-N}$ depends on $(t,s)$ in an admissible fashion. The amplitudes $\tilde{b}^\pm_j(t,s,y,x,\xi)$ have compact support with respect to $y$ and vanish for $\vert\xi\vert$ small. Moreover, $\tilde{b}_j^\pm$ and $\phi_j^\pm$ and their derivatives are uniformly  bounded for $s\in[0,T]$. A similar representation holds for $\psi V(t,s)$.
\end{Thm}
For proving Theorem 6 we start by looking for the properties of the adjoint operators $(\psi U(t,s))^*$ and $(\psi V(t,s))^*$.
\begin{lem}
 Let $\psi\in{\CI}$ and $\tau_1$, $\tau_2\in[0,T]$, with $\tau_1<\tau_2$. Then
\[(\psi V(\tau_2,\tau_1))^*=V(\tau_1,\tau_2)\psi,\quad (\psi\partial_tU(\tau_2,\tau_1))^*=\partial_tU(\tau_1,\tau_2)\psi.\]
\end{lem}
\begin{proof}
Let $f,g\in{\CI}$. Consider $u(t)=V(t,\tau_1)f$ and $w(t)=V(t,\tau_2)g$. Integrating by parts, we find
\[0=\int_{\tau_1}^{\tau_2}\int_{{\R}^n}(\partial_t^2u-\Div_x(a\nabla_x u)(t,x)w(t,x)\textrm{d}x\textrm{d}t=\int_{{\R}^n}[\partial_t(u)w]_{\tau_1}^{\tau_2}\textrm{d}x-\int_{{\R}^n}[\partial_t(w)u]_{\tau_1}^{\tau_2}\textrm{d}x.\]
It follows
\[0=\langle g,w(\tau_2)\rangle_{L^2}-\langle u(\tau_2),f\rangle_{L^2}=\langle g,V(\tau_2,\tau_1)f\rangle_{L^2}-\langle V(\tau_2,\tau_1)g,f\rangle_{L^2}.\]
Thus  for all $f,g\in{\CI}$ we have
\[\langle V(\tau_2,\tau_1)f,g\rangle_{L^2}=\langle f,V(\tau_1,\tau_2) g\rangle_{L^2}.\]
 By density we find that for all $f,g\in L^2$
\[\langle\psi V(\tau_2,\tau_1)f,g\rangle_{L^2}=\langle f,V(\tau_1,\tau_2)\psi g\rangle_{L^2}.\]
Following the same argument, we get
\[(\psi\partial_t U(\tau_2,\tau_1))^*=\partial_t(U)(\tau_1,\tau_2)\psi.\]
\end{proof}

We associate to the equation $\partial^2_tu-\Div_x(a\nabla_x u)=0$ the following Hamiltonian 
\[H(t,x,\tau,\xi)=\tau^2-a(t,x)\vert\xi\vert^2.\]
Thus the bicharacteristics of $\partial_t^2-\Div_x(a\nabla_x u)$ are the solutions of the system
\begin{equation}\label{eq:1}
\left\{\begin{array}{l} \displaystyle\frac{\partial x}{\partial \sigma}=-2a(t,x)\xi,\ \ \frac{\partial t}{\partial \sigma}=2\tau,\\
\ \\
       \displaystyle \frac{\partial\xi}{\partial \sigma}=\vert\xi\vert\nabla_xa(t,x),\ \frac{\partial\tau}{\partial \sigma}=a_t(t,x)\vert \xi\vert^2,\\
       \ \\
(t(0),x(0),\xi(0),\tau(0))=(t_0,x_0,\xi_0,\tau_0), \textrm{ with } \tau_0^2-a(t_0,x_0)\vert\xi_0\vert^2=0.\\
       \end{array}\right.
\end{equation}
\begin{prop}
 Let $(t_0,x_0,\tau_0,\xi_0)\in{\R}^{2(n+1)}$ be such that $\tau_0^2-a(t_0,x_0)\vert\xi_0\vert^2=0$ and $\xi_0\neq0$. Then the maximal solution of  (7.1) is defined for $\sigma\in I$, with $I$ an interval of ${\R}$, and $\sigma\mapsto t(\sigma)$ is a $C^\infty$ diffeomorphism from $I$ to ${\R}$.
\end{prop}
\begin{proof}
Let $(t(\sigma),x(\sigma),\xi(\sigma),\tau(\sigma))$ be the maximal solution of (7.1) defined in $I$. 
Now, suppose  there exists $\sigma_0\in I$ such that $\xi(\sigma_0)=0$. Since $(t_1(\sigma),x_1(\sigma),\tau_1(\sigma),\xi_1(\sigma))$ defined by
\[(t_1(\sigma),x_1(\sigma),\tau_1(\sigma),\xi_1(\sigma))=(t(\sigma_0),x(\sigma_0),0,0),\quad  \sigma\in{\R}\]
is the solution on ${\R}$ of (7.1) with initial data $(t(\sigma_0),x(\sigma_0),0,0)$, the uniqueness of the solution of (7.1) implies that \[(t(\sigma),x(\sigma),\tau(\sigma),\xi(\sigma))=(t_1(\sigma),x_1(\sigma),\tau_1(\sigma),\xi_1(\sigma)),\quad  \sigma\in{\R}.\]
This leads to a contradiction since $\xi(0)=\xi_0\neq0$. It follows that $\xi(\sigma)\neq0$, $  \sigma\in I$
and 
\[\sqrt{a(t(\sigma),x(\sigma))}\vert\xi(\sigma)\vert>0.\]
We deduce that $s\mapsto t(\sigma)$ is strictly monotonous in $I$ and
 $s\mapsto t(\sigma)$ is a $C^\infty$ diffeomorphism from $I$ to  $ran(t)$. Thus, for $\xi_0\neq0$ we can parametrize $(x(\sigma),\tau(\sigma),\xi(\sigma))$ with respect to $t$ and the problem (7.1) becomes
 \begin{equation}\label{eq:1}
\left\{\begin{array}{l} \displaystyle\frac{\partial x}{\partial t}=-\frac{a(t,x)\xi}{\tau(t)},\\
\ \\
       \displaystyle \frac{\partial\tau}{\partial t}=\frac{a_t(t,x)\vert \xi\vert^2}{2\tau(t)},\ \ \frac{\partial\xi}{\partial t}=\frac{\vert\xi\vert^2}{2\tau(t)}\nabla_xa(t,x),  \\
       \ \\
(x(t_0),\tau(t_0),\xi(t_0))=(x_0,\tau_0,\xi_0), \textrm{ with } \tau_0^2-a(t_0,x_0)\vert\xi_0\vert^2=0.\\
       \end{array}\right.\end{equation}
Let $\tau_0=\sqrt{a(t_0,x_0)}\vert \xi_0\vert$. 
We deduce that $\tau(t)=\sqrt{a(t,x(t))}\vert \xi(t)\vert$ and (7.2) becomes
\begin{equation}\label{eq=1}
\left\{\begin{array}{l}\displaystyle \frac{\partial x}{\partial t}=- \sqrt{a(t,x(t))}\frac{\xi(t)}{\vert\xi(t)\vert},\\
\ \\
\displaystyle\frac{\partial \xi}{\partial t}=\frac{\vert\xi(t)\vert}{2\sqrt{a(t,x(t))}}\nabla_xa(t,x(t)),\\
\ \\
(x(t_0),\xi(t_0))=(x_0,\xi_0),\quad\xi_0\neq0.\\ \end{array}\right.\end{equation}
 Let $(x(t),\xi(t))$ be the maximal solution of (7.3) defined on $J$. 
 For all $t\in J$ we have
 \[\frac{\partial\vert x(t)\vert^2}{\partial t}= 2\sqrt{a(t,x(t))}\frac{\xi(t).x(t)}{\vert\xi(t)\vert}\]
 and
 \[\frac{\partial \vert\xi(t)\vert^2}{\partial t}=\frac{\vert\xi(t)\vert}{\sqrt{a(t,x(t))}}\nabla_xa(t,x(t)).\xi(t).\]
 Let 
 \[C=\sup_{(t,x)\in{\R}^{1+n}}2\sqrt{a(t,x)}+\frac{\vert \nabla_xa(t,x)\vert}{\sqrt{a(t,x)}}<+\infty.\]
 We obtain
 \[\frac{\partial}{\partial t}\Vert (x(t),\xi(t))\Vert^2\leq C(1+\Vert (x(t),\xi(t))\Vert^2),\]
and it follows
 \begin{equation}\label{eq=1}
 \Vert (x(t),\xi(t))\Vert^2\leq \left(1+\Vert (x_0,\xi_0)\Vert^2\right)e^{C \vert t-t_0\vert},\quad  t\in J.\end{equation}
Also (7.2) implies
  \[\frac{\partial \vert\xi(t)\vert^2}{\partial t}=-\frac{\vert\xi(t)\vert^2}{a(t,x(t))}\nabla_xa(t,x(t)).\frac{\partial x(t)}{\partial t}.\]
 
It follows
 \begin{equation}\label{eq=1}
 \vert\xi(t)\vert^2=\vert\xi_0\vert^2\exp\left(-\int_{t_0}^t\frac{\nabla_xa(t',x(t')).\frac{\partial x(t')}{\partial t}}{a(t',x(t'))}\textrm{d}t'\right).\end{equation}
Since we have
 \[\int_{t_0}^t\frac{\nabla_xa(t',x(t')).\frac{\partial x(t')}{\partial t}}{a(t',x(t'))}\textrm{d}t'=\ln\left(\frac{a(t,x(t))}{a(t_0,x(t_0))}\right)-\int_{t_0}^t\frac{a_t(t',x(t'))}{a(t',x(t'))}\textrm{d}t',\]
 we obtain
  \begin{equation}\label{eq=1}\vert\xi(t)\vert^2=\vert\xi_0\vert^2\left(\frac{a(t_0,x_0)}{a(t,x(t))}\right) \exp\left(\int_{t_0}^t\frac{a_t(t',x^(t'))}{a(t',x(t'))}\textrm{d}t'\right).\end{equation}
 Let 
 \[D=\sup_{(t,x)\in{\R}^{1+n}}\frac{\vert a_t(t,x)\vert}{a(t,x)}<+\infty.\]
We deduce from (7.6) and condition (i) of (1.2) that
\begin{equation}\label{eq=1}\vert\xi(t)\vert^2\geq\vert\xi_0\vert^2\left(\frac{c_0}{C_0}\right) e^{-D\vert t-t_0\vert},\quad   t\in J.\end{equation}
Conditions (7.4) and (7.7) imply  that $J={\R}$ and it follows that the maximal  solution $(x(t),\tau(t),\xi(t))$ of (7.2) is defined on ${\R}$.
It follows the same for $\tau_0=-\sqrt{a(t_0,x_0)}\vert \xi_0\vert$. 

Now let  
 $(x_1(t),\tau_1(t),\xi_1(t))$ be the  solution on ${\R}$ of (7.2) and $(t(\sigma),x(\sigma),\tau(\sigma),\xi(\sigma))$ the solution of (7.1) on ${\R}$. Let $s_1(t)$ be a function defined by 
 \[s_1(t)=\int_{t_0}^t\frac{\textrm{d}t'}{2\tau_1(t')},\]
 and notice that $s_1(t)$ is a diffeomorphism from ${\R}$ to $ran(s_1)$. Consider also $(t_3(\sigma), x_3(\sigma),\tau_3(\sigma),\xi_3(\sigma))$ defined on $ran(s_1)$  by
 \[(t_3(\sigma), x_3(\sigma),\tau_3(\sigma),\xi_3(\sigma))=(s_1^{-1}(\sigma),x_1(s_1^{-1}(\sigma)),\tau_1(s_1^{-1}(\sigma)),\xi_1(s_1^{-1}(\sigma))).\]
 We can easily see that $(t_3(\sigma), x_3(\sigma),\tau_3(\sigma),\xi_3(\sigma))$ is the solution on $ran(s_1)$ of (7.1). Thus  the uniqueness of the solution of (7.1) implies that for all $s\in ran(s_1)$
 \[(t_3(\sigma), x_3(\sigma),\tau_3(\sigma),\xi_3(\sigma))=(t(\sigma), x(\sigma),\tau(\sigma),\xi(\sigma)).\]
 Then, it follows that for all $h\in{\R}$
 \[t(s_1(h))=s_1^{-1}(s_1(h))=h\]
 and the range of $s\mapsto t(\sigma)$ is ${\R}$. In conclusion, $s\mapsto t(\sigma)$ is a $C^\infty$ diffeomorphism from $I$ to ${\R}$ and this completes the proof.
 
 \end{proof}

Let $P=\partial_t^2-\Div_x(a(t,x)\nabla_x)$ and let $R_1>\rho>0$.
\begin{prop}
Let $(t_0,x_0,\eta_0)\in[0,T]\times B(0,R_1)\times S^{n-1}$, $R_1>0$ and let $S(y,\xi)$ be a $C^\infty$ function with support in a sufficiently small neighborhood of $(x_0,\eta_0)$. There exist $\delta_{t_0}>0$, $r_{x_0}>0$ and a neighborhood $\omega_{\eta_0}\subset S^{n-1}$ of $\eta_0$ such that for every integer $N\geq1$ we can construct Fourier integral operators
\[(I^\pm(t)f)(t,x)=\int\int e^{i\phi^\pm(t,x,\xi)-iy.\xi}b^{\pm}(t,x,y,\xi)f(y)\textrm{d}y\textrm{d}\xi\]
so that 
\begin{equation}\label{eq=1}
\left\{\begin{array}{lll}
P(I^+(t)f+I^-(t)f)&=&R_N(t)f,\\
I^+(t_0)f+I^-(t_0)f&=& S(x,D_x)f+V_Nf,\\
\partial_t I^+(t_0)f+\partial_tI^-(t_0)f&=&W_Nf.
\end{array}\right.
\end{equation}
Here $V_N, W_N\in B^{-N}$ and $R_N(t)\in B^{-N}$ for $\vert t-t_0\vert<\delta_{t_0}$ and the dependence on $t$ is in an admissible fashion. Moreover, $R_N(t)$, $V_N$ and $W_N$ are bounded  uniformly with respect to $t_0$, $x_0$, $\eta_0$. In a similar way we can construct Fourier integral operators
\[(J^\pm(t)f)(t,x)=\int\int e^{i\phi^\pm(t,x,\xi)-iy.\xi}c^{\pm}(t,x,y,\xi)f(y)\textrm{d}y\textrm{d}\xi\]
with the properties 
\[\left\{\begin{array}{lll}
P(J^+(t)f+J^-(t)f)&=&R'_N(t)f,\\
J^+(t_0)f+J^-(t_0)f&=& V'_Nf,\\
\partial_t J^+(t_0)f+\partial_tJ^-(t_0)f&=&S(x,D_x)f+W'_Nf.
\end{array}\right.\]
$R'_N(t)$, $V'_N$ and $W'_N$ being smoothing operators having the same properties as $R_N(t)$, $V_N$ and $W_N$.

\end{prop}
\begin{proof}
We will solve problem (7.8) using BKW method. First, $P$ is a strictly hyperbolic operator with principal symbol
\[\sigma(P)=\tau^2-a(t,x)\vert\xi\vert^2=(\tau-\sqrt{a(t,x)}\vert\xi\vert)(\tau+\sqrt{a(t,x)}\vert\xi\vert).\]
Consider the eikonal equations
\begin{equation}\label{eq=1}
\left\{\begin{array}{l}\displaystyle\frac{\partial \phi^\pm}{\partial t}\pm \sqrt{a(t,x)}\vert\nabla_x\phi^\pm\vert=0,\\
\ \\
\displaystyle\phi^\pm(t_0,x,\eta)=x.\eta.\\ \end{array}\right.\end{equation}
We can solve  (7.9) for $(t,x,\eta)\in[t_0-\delta_{t_0},t_0+\delta_{t_0}]\times B(x_0,r_{x_0})\times \omega_{\eta_0}$ by the classical Hamilton-Jacobi method (see [9], chapter V). We  extend $\phi^\pm$ to a smooth solution  which  is a positively homogeneous of degree 1 in $\eta$, by putting $\phi(t,x,\eta)=\vert\eta\vert\phi(t,x,\frac{\eta}{\vert\eta\vert})$. Let $\eta'\in S^{n-1}$. Consider
\[\Lambda^\pm_1(t_0,\eta')=\{(t_0,x,\mp\sqrt{a(t_0,x)},\eta')\ :\ \vert x\vert\leq R\}.\]
$\Lambda^\pm_1(t_0,\eta')$ is a $C^\infty$-submanifold of dimension $n$.
We introduce the Lagrangians 
\[\begin{array}{lll}\Lambda^\pm(t_0,\eta')=\{(t,x(t),\tau(t),\xi(t))&:&\ t\in[0,T+1],\ (t(\sigma),x(t(\sigma)),\tau(t(\sigma)),\xi(t(\sigma))) \textrm{ be the}  \\ \ &\ &\textrm{bicharacteristic with initial data  }\\ \ &\ &(t(0),x(t(0),\tau(t(0)),\xi(t(0)))\in\Lambda^\pm_1(t_0,\eta')\}.\\ \end{array}\]
 Let $(t^\pm(\sigma),x^\pm(\sigma),\tau^\pm(\sigma),\xi^\pm(\sigma))$ be the bicharacteristic with initial data
\[(t^\pm(0),x^\pm(0),\tau^\pm(0),\xi^\pm(0))=(t_0,x_0,\mp\sqrt{a(t_0,x_0)},\eta'),\]
with  $\vert x_0\vert\leq R_1$.
By the theory for the Hamilton-Jacobi equation we deduce
\[\left(\frac{\partial\phi^\pm}{\partial t},\nabla_x\phi^\pm\right)(t^\pm(\sigma),x^\pm(\sigma),\eta')=(\tau^\pm(\sigma),\xi^\pm(\sigma)).\]
Then $\Lambda^\pm(t_0,\eta')$ is the graph of  $\phi^\pm$ for $(t,x,\eta')\in[t_0-\delta_{t_0},t_0+\delta_{t_0}]\times B(x_0,r_{x_0})\times\omega_{\eta_0}$.  Now consider $\left(x^\pm(t,t',x,\eta'),\xi^\pm(t,t',x,\eta'),\tau^\pm(t,t',x,\eta')\right)$  the solution on ${\R}$ of the problem (7.2) with initial data\\
 $x^\pm(t',t',x,\eta')=x$,
 $\xi^\pm(t',t',x,\eta')=\eta'$ and $\tau^\pm(t',t',x,\eta')=\mp\sqrt{a(t',x)}\vert\eta'\vert$. We can see that $x^\pm(t,t',x,\eta')$ and $\xi^\pm(t,t',x,\eta')$ are continuous with respect to $t,t',x,\eta'$ on the compact set\\ 
 $[0,T]\times[0,T]\times B_F(0,R)\times S^{n-1}$ and there exists $C>0$ such that
\[\vert x(t,t',x,\eta')\vert\leq C,\quad \vert\xi(t,t',x,\eta')\vert\leq C,\quad  (t,t',x,\eta')\in[0,T]\times[0,T]\times B_F(0,R)\times S^{n-1}.\]
It follows that for all $t_0\in[0,T]$ and $\eta'\in S^{n-1}$ the Lagrangian $\Lambda^\pm(t_0,\eta')$ is included in a set bounded uniformly with respect to $(t,x,\eta)\in[t_0-\delta_{t_0},t_0+\delta_{t_0}]\times B(x_0,r_{x_0})\times \omega_{\eta_0}$ independently of $t_0$, $x_0$, $\eta_0$.
Consequently,  $\phi^\pm$ and  their derivatives are uniformly bounded with respect to\\
 $(t,x,\eta)\in[t_0-\delta_{t_0},t_0+\delta_{t_0}]\times B(x_0,r_{x_0})\times \omega_{\eta_0}$ independently of $t_0$, $x_0$, $\eta_0$.

By using a cut-off, we may assume that all symbols vanish for sufficiently small $\vert \eta\vert$. Now consider the asymptotic expansion of $b^\pm$
\[b^\pm \sim\sum_{k=0}^Nb_k^\pm\]
with $b_k^\pm$ homogeneous of degree $-k$ in $\eta$. To solve (7.8) (see [9] chapter VI), $b_0^\pm$ must be  solutions of the transport equation
\begin{equation}\label{eq=1}\left\{\begin{array}{l}
L^\pm(b_0^\pm)=0,\\
b^+_0(t_0,x,y,\eta)+b^-_0(t_0,x,y,\eta)=S(y,\eta)\\
\left(\partial_t\phi^+ b^+_0+\partial_t\phi^- b^-_0\right)(t_0,x,\eta)=0\\
\end{array}\right.\end{equation}
and $b^\pm_k$ for $k\in\{1,\cdots,N\}$ must be solutions of the transport equation
\begin{equation}\label{eq=1}\left\{\begin{array}{l}
L^\pm(b_k^\pm)=-P(b_{k-1}^\pm),\\
b^+_k(t_0,x,\eta)+b^-_k(t_0,x,\eta)=0\\
\left(\partial_t\phi^+ b^+_k+\partial_t\phi^- b^-_k\right)(t_0,x,\eta)=-\left(\partial_t b_{k-1}^++\partial_t b_{k-1}^-\right)(t_0,x,\eta)\\
\end{array}\right.\end{equation}
with
\[L^\pm(v)=\partial_t\phi^\pm\partial_tv-2a\nabla_x\phi^\pm.\nabla_xv+P(\phi^\pm)v.\]
Since 
\[\frac{\partial\phi^+}{\partial t}(t_0,x,\eta)=-\sqrt{a(t,x)}\vert \nabla_x \phi(t_0,x,\eta)\vert=-\sqrt{a(t,x)}\vert \eta\vert\]
and
\[\frac{\partial\phi^-}{\partial t}(t_0,x,\eta)=\sqrt{a(t,x)}\vert \eta\vert\]
we have
\[\frac{\partial\phi^-}{\partial t}(t_0,x,\eta)\neq\frac{\partial\phi^+}{\partial t}(t_0,x,\eta).\]
It follows that we can find $b_0^\pm$ and $b^\pm_k$ for $k\in\{1,\cdots,N\}$ as solutions of (7.10) and (7.11) on $[t_0-\delta_{t_0},t_0+\delta_{t_0}]\times B(x_0,r_{x_0})\times \omega_{\eta_0}$. Moreover, in a sufficiently small neighborhood of $(x_0,\eta_0)$  ,   $b^\pm_k$ will satisfy $\textrm{supp}_x(b^\pm_k(t,x,y,\eta))\subset B(x_0,r_{x_0})$ and $\textrm{supp}_\eta(b_k^\pm(t,x,y,\eta))\subset \omega_{\eta_0}$, and $b_0^\pm$, $b^\pm_k$ for $k\in\{1,\ldots,N\}$  solutions of (7.10) and (7.11) on $[t_0-\delta_{t_0},t_0+\delta_{t_0}]\times{\R}^n_x\times{\R}^n_\eta$ homogeneous of degree $-k$ in $\eta$.
Let $(t^\pm(\sigma),x^\pm(\sigma),\tau^\pm(\sigma),\xi^\pm(\sigma))$ be the bicharacteristic with initial data
\[(t^\pm(0),x^\pm(0),\tau^\pm(0),\xi^\pm(0))=(t_0,x,\mp\sqrt{a(t_0,x)},\frac{\eta}{\vert\eta\vert}).\]
Then $b_0^\pm(t(\sigma),x(\sigma),\frac{\eta}{\vert\eta\vert})$ is the solution of
\[\left\{\begin{array}{l}\partial_\sigma(b^\pm_0(t(\sigma),x(\sigma),\frac{\eta}{\vert\eta\vert}))+P(\phi)b_0^\pm(t(\sigma),x(\sigma),\frac{\eta}{\vert\eta\vert})=0,\\
b_0^\pm(t(\sigma),x(\sigma),\frac{\eta}{\vert\eta\vert})_{\vert \sigma=0}=b^\pm_0(t_0,x,\frac{\eta}{\vert\eta\vert}).
\end{array}\right.\]
Thus
\[b_0^\pm(t(\sigma),x(\sigma),\eta)=b_0^\pm(t_0,x,\eta)e^{-\int_0^\sigma P(\phi^\pm)(t(s'),x(s'),\frac{\eta}{\vert\eta\vert})\textrm{d}s'}\]
and since $\phi^\pm$ and all their derivatives are uniformly bounded independently  of $t_0$, $x_0$ and $\eta_0$ it follows the same for $b_0^\pm$. In the same way we prove that $b^\pm_1,\ldots,b_N^\pm$ are uniformly bounded on $[t_0-\delta_{t_0},\delta_{t_0}+t_0]\times B(x_0,r_{x_0})\times{\R}^n$ independently of $t_0$, $x_0$ and $\eta_0$. Finally, we find that
\[\begin{array}{l}(\partial_t^2-\Div_x(a(t,x)\nabla_x))I(t,x,y,t_0,x_0)f=\\
=\int_{{\R}^n}\int_{{\R}^n}(P(b_N^+)e^{i\phi^+}+P(b_N^-)e^{i\phi^-})(t,x,\eta)e^{-iy.\eta}f(y)\textrm{d}\eta \textrm{d}y  \end{array}\]
and
\[I(t_0,t_0,x,y)=\int_{{\R}^n}(S(y,\eta)+V_N(x,y,\eta))e^{i(x-y).\eta}\textrm{d}\eta,\]
\[\partial_tI(t_0,t_0,x,y)=\int_{{\R}^n}W_N(x,y,\eta)e^{i(x-y).\eta}\textrm{d}\eta\]
with $V_N(x,y,\eta),W_N(x,y,\eta)\in S^{-N}_{1,0}$.
The proof is complete since\\
 $(P(b_N^+)e^{i\phi^+}+P(b_N^-))\in S^{-N}_{1,0}$  is uniformly bounded on $[t_0-\delta_{t_0},t_0+\delta_{t_0}]\times{\R}^n_x\times{\R}^n_\xi$ independently of $t_0$,$x_0$ and $\eta_0$. We apply the same argument for $J^\pm(t)$.
\end{proof}

\begin{lem}
Let $s_1,s_2\in[0,T]$. For all $t,s\in[s_1,s_2]$, $(\mathcal{U}(t,s))_1\in B^0$  depends on $t,s$ in an admissible fashion.
\end{lem}
Lemma 6 follows from the  properties of the  solutions of strictly hyperbolic equations.

\begin{prop}
Let $\psi\in\mathcal{C}^\infty_0(\vert x\vert\leq R_1)$, $R_1>\rho$. Then there exists $\delta>0$ such that   for $s,t\in[0,T]$ with $\vert s-t\vert<\delta$and every integer $N\geq1$ we have the representation

\[U(t,s)\psi=\sum_{j=1}^M(\tilde{I}_j^+(t,s)+\tilde{I}_j^-(t,s))+ R_N(t,s),\]
where $\tilde{I}_j^\pm(t,s)$ are Fourier integral operators with kernels
\[\tilde{I}^\pm_k(t,s,x,y)=\int b_k^\pm(s,t,x,y,\xi)e^{i\phi^\pm_k(s,t,x,\xi)-iy.\xi}\textrm{d}\xi\]
and $R_N(s,t)\in B^{-N}$ depends on $(t,s)$ in an admissible fashion. The amplitudes $b^\pm_j(t,s,x,y,\xi)$ have compact support with respect to $x$ and vanish for $\vert\xi\vert$ small. Moreover, $b_j^\pm$ and $\phi_j^\pm$ and their derivatives are uniformly bounded for $s\in[0,T]$, and  $\phi_k^\pm(t,s,x,\xi)$ is the solution on $[s-\delta,s+\delta]\times \textrm{supp}_{(y,\xi)}(b^\pm_k)\cup \textrm{supp}_{(y,\xi)}(c^\pm_k)$ homogeneous in $\xi$ of the eikonal equation
\begin{equation}\label{eq=1}\left\{ \begin{array}{l} \partial_s(\phi_k^\pm)(s,t,x\xi)\pm\sqrt{a(t,x)}\vert\nabla_x\phi^\pm_k(t,s,x,\xi)\vert=0,\\
\phi^\pm_k(t,t,x,\xi)=x.\xi.\end{array}\right.\end{equation}
 A similar representation holds for $ V(s,t)\psi$..
\end{prop}
\begin{proof}
 Let $t_0\in[0,T]$ and $s,\tau\in[t_0-\delta_{t_0},t_0+\delta_{t_0}]$ (with $\delta_{t_0}$ as in Proposition 9). Consider  $R>0$ such that supp$\psi\subset\{x\ :\ \vert x\vert\leq R\}=B_F(0,R)$. Since $B_F(0,R)\times S^{n-1}$ is compact, Proposition 9 implies that  for $\delta_{t_0}$ sufficiently small we can find symbols $S_1(y,\xi),\ldots,S_M(y,\xi)$ such that:\\
 
 (i) $S_1(y,\xi),\ldots,S_M(y,\xi)$ are homogeneous of degree 1 in $\xi$,\\
  $S_1(y,\xi),\ldots,S_M(y,\xi)\in\mathcal{C}^\infty_0(B_F(0,R)\times S^{n-1})$ and
 \[\sum_{i=1}^MS_i(y,\xi)=\psi(y),\]

 (ii) we can find Fourier integral operators $I^\pm_1(s,\tau),\ldots,I_M^\pm(s,\tau)$ constructed in Proposition 9  such that  $I^+_1(s,\tau)+I^-_1(s,\tau),\ldots,I^+_M(s,\tau)+I_M^-(s,\tau)$ are respectively the solutions of (7.8) on $[t_0-\delta_{t_0},t_0+\delta_{t_0}]\times{\R}^n_x$ with $S(y,\xi)$ replaced respectively by $S_1(y,\xi),\ldots,S_M(y,\xi)$ and $t,t_0$ replaced by $s,\tau$.
 
 Thus Lemma 6 and Proposition 9 implies that for $\delta_{t_0}$ small enough we have
\[U(s,\tau)\psi=\sum_{i=1}^MI_i^\pm(s,\tau,t_0) +Q_N(s,\tau,t_0),\]
where $I_i^\pm(s,\tau,t_0)$ have  kernels
\[I_i^\pm(s,\tau,t_0,x,y)=\int_{{\R}^n}b^\pm_i(s,\tau,x,y,\xi)e^{i\phi^\pm_i(s,t,x,\xi)-ix.\xi}\textrm{d}\xi\]
with $b^\pm_i$, $\phi_i^\pm$ and all their derivatives  bounded independently on $s,\tau$, while\\
 $Q_N(s,\tau,t_0)\in B^{-N}$   depends  on $s$, $\tau$ in admissible fashion. 

We are now able  to establish Proposition 10. We have
\[[0,T]\subset\bigcup_{t_0\in[0,T]}[t_0-\frac{\delta_{t_0}}{3},t_0+\frac{\delta_{t_0}}{3}]\]
and since $[0,T]$ is compact, it follows that there exist $t_1,\cdots,t_l\in[0,T]$ such that
\[[0,T]\subset\bigcup_{i=1}^l\ [t_i-\frac{\delta_{t_i}}{3},t_i+\frac{\delta_{t_i}}{3}].\]
Consider $\delta=\min\{\frac{\delta_{t_1}}{3},\cdots,\frac{\delta_{t_l}}{3}\}$. Let $s,t\in[0,T]$ be such that $\vert s-t\vert<\delta$. Then there exists $j\in\{1,\cdots,l\}$ such that $t\in[t_i-\frac{\delta_{t_i}}{3},t_i+\frac{\delta_{t_i}}{3}]$, and
\[\vert t_i-s\vert\leq\vert t-s\vert+ \vert t-t_i\vert\leq \frac{2\delta_{t_i}}{3}<\delta_{t_i}.\]
Thus 
\[U(s,t)\psi=\sum_{j=1}^{M_{t_i}}I_j^\pm(s,t,t_j)+Q_N(s,t,t_j).\]
where  $I_j^\pm(s,t,t_i)$ and $Q_N(s,t,t_i)$  have the same properties as described in Proposition 9. We treat $V(s,t)\psi$ in the same way.
\end{proof}
\textit{Proof of Theorem 6:}
Lemma 5 and Proposition 10 imply that
\[\psi V(t,s)=(V(s,t)\psi)^*=\sum_{j=1}^M(J^\pm_j(s,t))^*+ (R_N(s,t))^*.\]
with $J^\pm_j(s,t)$ a Fourier integral operator having the following kernel
\[J^\pm_j(s,t)=\int c^\pm_j(s,t,x,y,\xi)e^{i\phi^\pm_j(s,t,x,\xi)-iy.\xi}\textrm{d}\xi\]
and  $R_N(s,t)\in B^{-N}$ is an operator which depends on $s$, $t$ in an admissible fashion.
Choose
\[\tilde{J}_j^\pm(t,s)=(J^\pm_j(s,t))^*\]
and
\[\tilde{R}_N(t,s))=(R_N(s,t))^*.\]
Then $\tilde{J}^\pm_j(t,s)$ will be a Fourier integral operator with kernel
\[\tilde{J}^\pm_j(t,s,x,y)=\int \tilde{c}^\pm_j(t,s,y,x,\xi)e^{-i\tilde{\phi}_j^\pm(t,s,y,\xi)+ix.\xi}\textrm{d}\xi\]
with $\tilde{c}^\pm_j(t,s,y,x,\xi)=\overline{c^\pm_j(s,t,y,x,\xi)}$ and $\tilde{\phi}_j^\pm(t,s,y,\xi
)=\phi_j^\pm(s,t,y,\xi)$. Finally, $\tilde{R}_N(t,s)$ will satisfy the same properties of regularity as $R_N(s,t)$.
The same holds for $\psi U(t,s)$.$\quad\quad\quad\quad\quad\quad\quad\quad\quad\quad\square$

\subsection{Besov spaces}

Here we recall some results for Besov spaces. We start with the construction of these spaces. Consider $\chi\in{\CI}$ such that
 supp$(\chi)=\{\xi\in{\R}^n:\half\leq\vert\xi\vert\leq2\}$, with $\chi(\xi)>0$ for $\half<\vert\xi\vert<2$. Assume that $\chi$ satisfies  
\[\sum_{k=1}^{+\infty}\chi(2^{-k}\xi)=1\quad\textrm{for}\quad \xi\neq0.\]
 For $j\geq1$, consider $\chi_j(\xi)=\chi(2^{-j}\xi)$ and $\chi_0(\xi)=1-\sum_{j=1}^{+\infty}\chi_j(\xi)$. Let $s\in{\R}$ and\\
 $1\leq p,q\leq\infty$. We define $\Vert.\Vert_{B^s_{p,q}}$ by
\[\Vert f\Vert_{B^s_{p,q}({\R}^n)}=\Vert (2^{js}\Vert \chi_j(D)f\Vert_{L^p({\R}^n)} )_{j\in\mathbb{N}}\Vert_{l^q(\mathbb{N})}.\]
The  spaces 
\[B^s_{p,q}({\R}^n)=\{u\in S'({\R}^n):\   \Vert u\Vert_{B^s_{p,q}({\R}^n)}<+\infty\}\]
are called Besov spaces.
Here $\chi_j(D)f=\mathcal{F}^{-1}(\chi_j(\xi)\widehat{f}(\xi))$. Now we recall some properties of Besov spaces.
\begin{prop}
  Let $1\leq p,q<\infty$ and $s\in{\R}$. Consider $p',q'$ defined by $\frac{1}{p}+\frac{1}{p'}=1$ and $\frac{1}{q}+\frac{1}{q'}=1$. Then  $B^{-s}_{p',q'}({\R}^n)$ is the dual space of $B^s_{p,q}({\R}^n)$.
\end{prop}
\begin{prop}
 Let $2\leq q\leq\infty$. Then 
\[\Vert u\Vert_{L^q({\R}^n)}\leq C\Vert u\Vert_{B^0_{q,2}({\R}^n)},\quad  u\in B^0_{q,2}({\R}^n).\]
\end{prop}
\begin{prop}
 Let $A\in B^m$. Then for all $s\in{\R}$, $1\leq p,q\leq\infty$ we get\\ $A\in\mathcal{L}(B^s_{p,q},B^{s-m}_{p,q})$.
\end{prop}

We refer to [12] and [26] for the proof of those properties.

\subsection{Proof of Theorem 5}

Consider  the operators $\tilde{I}_j(t,s)$ which approximate   $\psi U(t,s)$ for\\
 $0<t-s<\delta$.
We start with the following result of Kapitanski.
\begin{lem} 
  Let $\tilde I(t,s)$ be a Fourier integral operator with kernel 
\[\tilde{I}(t,s,x,y)=\int_{{\R}^n}b(t,s,y,x,\xi)e^{ix.\xi-i\phi(t,s,y,\xi)}\textrm{d} \xi.\]
Suppose that $b(t,s,x,\xi)\in S^0_{1,0}$ is such that $b(t,s,.,.)$ depends smoothly on $t,s$,\\
 $\textrm{supp}_y b(t,s,y,x,\xi)\subset\{y\in{\R}^n:\vert y\vert\leq R\}$,
 $b(t,s,y,x,\xi)=0$ for small $\vert\xi\vert$, while  $\phi$ is $C^\infty$ and homogeneous of degree 1 in $\xi$ with
  $\phi(s,s,y,\xi)=y.\xi$. Let $r,\epsilon>0$ and let $m$ be the maximum positive integer such that 
\[m\leq rank(\partial_\xi^2\partial_t\phi(s,s,y,\xi)),\ \ \vert y\vert\leq r,\ \ \vert\xi\vert\geq \epsilon .\]
Then for $\vert t-s\vert$ sufficiently small and $2\leq q<\infty$, $\nu\in{\R}$ satisfying
\[\left(n-\frac{m}{2}\right)\frac{q-2}{q}<\nu<\frac{n(q-2)}{q},\]
we have 
\[\Vert\tilde I(t,s)\Vert_{B^0_{q,2}({\R}^n)}\leq C\vert t-s\vert^{\nu-\frac{n(q-2)}{q}}\Vert f\Vert_{B^\nu_{q',2}}.\]

\end{lem}
Notice that this lemma  is a generalization of Lemma 3.1  in [12], but the proof is the same since the phase $\phi$ does not depend on $x$.

In our case we know that
\[\partial_t(\phi^\pm(s,s,y,\xi))=\mp\sqrt{a(s,y)}\vert\nabla_y\phi^\pm(s,y,\xi)\vert,\]
and $\phi^\pm(s,s,y,\xi)=y.\xi$. Thus, $\partial_t\phi^\pm(s,s,y,\xi)=\mp\sqrt{a(s,y)}\vert\xi\vert$ and we obtain that
\[\partial_\xi^2\partial_t\phi^\pm(s,s,y,\xi)=\mp\sqrt{a(s,y)}\partial_\xi^2(\vert\xi\vert).\]
\begin{prop}
Let $n\geq3$ and let $g$ be a function defined on ${\R}^n$ by $g:\xi\longmapsto\vert\xi\vert$. Then

 \[rank(\partial_\xi^2(g))(\xi)=n-1,\quad \xi\neq0.\]
\end{prop}

 It follows from Proposition 14 that for $\tilde{I}_j(t,s)$ we have $m=n-1$. 
\begin{prop}
Consider $2\leq q<\infty$ and $\frac{(n+1)(q-2)}{2q}<\nu<\frac{n(q-2)}{q}$ and a cut-off function $\psi\in{\CI}$. Then there exists $\delta>0$ such that for all
 $s,t\in[0,T]$, $0<t-s<\delta$ we have
\[\Vert \psi U(t,s)f\Vert_{B^0_{q,2}({\R}^n)}\leq C\vert t-s\vert^{\nu-\frac{n(q-2)}{q}}\Vert f\Vert_{B^\nu_{q',2}({\R}^n)},\]
\[\Vert \psi V(t,s)g\Vert_{B^0_{q,2}({\R}^n)}\leq C\vert t-s\vert^{\nu-\frac{n(q-2)}{q}}\left\Vert \frac{g}{\Lambda}\right\Vert_{B^\nu_{q',2}({\R}^n)},\]
with $C>0$  independent on $s$, $t$ and $f$.
\end{prop}
Kapitanski established  the result of Proposition 15 for $s=0$, in Theorem 1 of [12],  by applying Lemma 7  to the representation of the propagator with Fourier integral operators in a small neighborhood of $t=0$. In Theorem 6 we have shown that we can represent  $\psi U(t,s)$ and $\psi V(t,s)$ with a sum of Fourier integral operators with amplitude and phase uniformly bounded independently of $s,t\in[0,T]$, and a sufficiently regular operator bounded independently of $s,t\in[0,T]$. With this  argument we can apply the result of Kapitanski to obtain Proposition 15.

\begin{Thm}
 Let  $\psi\in{\CI}$. Then  for  $2\leq p,q<+\infty$ and $\gamma>0$ satisfying
\begin{equation}\label{eq:1}\frac{n(q-2)}{2q}-\gamma=\frac{1}{p}\leq\frac{(n-1)(q-2)}{4q},\end{equation}
there exists $\delta>0$ such that for all $s\in[0,T]$
 \[\int_s^{s+\delta} \Vert\psi (\mathcal{U}(t,s)f)_1\Vert_{B^0_{q,2}({\R}^n)}^p\textrm{d}t\leq C(T,\phi,p,q,n)\Vert f\Vert_{\dot{\mathcal{H}}_\gamma}^p\]
with $C>0$ independent on $s$, $f$.
\end{Thm}
Applying Proposition 15, the proof of  Theorem 7 is similar to the proof  of Theorem 2 in [12], but we must replace $U(t,s)$ by $\psi U(t,s)$ and change the definition given for $\Lambda$ by Kapitanski. It follows the same for $\psi V(t,s)$.
Theorem 5 follows directly from Theorem 7 and  Proposition 13.\\
Notice that we can establish, with the same method, a local Strichartz estimates without assuming that  $a(t,x)$ is periodic and $n\geq3$ is odd. Indeed, we obtain the following
\begin{cor}
 Assume $n\geq3$ and  $a(t,x)$ is a $C^\infty$ function on ${\R}^{n+1}$ satisfying conditions (i) and (ii) of (1.2). Let $2\leq p,q<+\infty$, $\gamma>0$ be such that
\begin{equation}\label{eq:1}\frac{1}{p}=\frac{n(q-2)}{2}-\gamma,\quad\frac{1}{p}\leq\frac{(n-1)(q-2)}{4q}.\end{equation}
Then   for all $u$ solution of (1.1) with $\tau=0$ we have
\[\Vert u\Vert_{L^p([0,\delta],L^q({\R}^n))}+ \Vert u(t)\Vert_{\mathcal{C}([0,\delta],\dot{H}^\gamma)}+\Vert \partial_t(u)(t)\Vert_{\mathcal{C}([0,\delta],\dot{H}^{\gamma-1})}\leq C(p,q,\rho,n)\Vert f\Vert_{\dot{\mathcal{H}}_\gamma}.\]
\end{cor}
\begin{proof}
Theorem 5 implies that
\[\Vert\chi u\Vert_{L^p([0,\delta],L^q)}\leq C\Vert f\Vert_{\dot{\mathcal{H}}_\gamma}.\]
Also, applying the local Strichartz estimates for the free wave equation (see Section 3) and the continuity with respect to $t$, we get
\[\Vert (1-\chi)u\Vert_{L^p([0,\delta],L^q)}\leq C(\delta)\Vert f\Vert_{\dot{\mathcal{H}}_\gamma}.\]
It follows that
\[\Vert  u\Vert_{L^p([0,\delta],L^q({\R}^n))}\leq C\Vert f\Vert_{\dot{\mathcal{H}}_\gamma}.\]
\end{proof}

\section{ Examples of Non-trapping Metrics $a(t,x)$ Satisfying Conditions (H1) and (H2)}
\renewcommand{\theequation}{\arabic{section}.\arabic{equation}}
\setcounter{equation}{0}
In this section we will give some examples of metrics $a(t,x)$. We start with a class of metrics  non-trapping $a(t,x)$.

\subsection{ Examples of non-trapping metric $a(t,x)$}

Consider $a(t,x)$ satisfies
\begin{equation}\label{eq:hyp1}
\frac{2a}{\rho} -\frac{\vert a_t\vert}{\sqrt{\inf a}}-\vert a_r\vert\geq\beta>0.
\end{equation}
We will show that if condition (8.1) is fulfilled, then $a(t,x)$ is a non-trapping metric.
We recall that the  bicharacteristics  $(t(\sigma),x(\sigma),\tau(\sigma),\xi(\sigma))$ of $\partial_t^2u-\Div_x(a\nabla_x u)=0$  are solutions of
\[\left\{\begin{array}{l}
\displaystyle\frac{\partial t}{\partial \sigma}=2\tau,\ 
\frac{\partial x}{\partial \sigma}=-2a(t,x)\xi,\\
\ \\
\displaystyle\frac{\partial \tau}{\partial \sigma}=a_t(t,x)\vert\xi\vert^2,\ 
\frac{\partial \xi}{\partial \sigma}=\vert\xi\vert^2\nabla_xa(t,x),\\
\ \\
   (x(0),t(0),\tau(0),\xi(0))=(x_0,t_0,\tau_0,\xi_0), \\     \end{array}\right.\]
with $H(t_0,x_0,\tau_0,\xi_0)=0$. We take $\vert x_0\vert\leq\rho$ and $\xi_0\neq0$.
In Proposition 8 we have established that  $s\longmapsto t(\sigma)$ is a diffeomorphism of ${\R}$ and $\tau(\sigma)=\pm\sqrt{a(t(\sigma),x(\sigma))}$. We find
\[\frac{\partial \vert x\vert^2}{\partial t}=\frac{\frac{\partial \vert x\vert^2}{\partial s}}{\frac{\partial t}{\partial s}}=-2a(t,x)\frac{\xi.x}{\tau}.\]
Also
\[\begin{array}{lll}\frac{\partial \left(\frac{\xi.x}{\tau}\right)}{\partial t}&=&\frac{\vert\xi\vert^2\nabla a.x}{2\tau^2}-\frac{a\vert\xi\vert^2}{\tau^2}+\frac{\xi.x}{2\tau}\left(-\frac{a_t\vert\xi\vert^2}{\tau^2}\right)\\
  \  &=&\frac{a_r}{2a}\vert x\vert-1-\frac{a_t}{2a}\frac{\xi.x}{\tau}.\\
  \end{array}\]
Since $a(t,x)=1$ for $\vert x\vert>\rho$ we have $a_t(t,x)=a_r(t,x)=0$ for $\vert x\vert>\rho$. It follows that
\[\vert a_t(t,x)\vert\vert x\vert\leq \vert a_t(t,x)\vert\rho,\quad\vert a_r(t,x)\vert\vert x\vert\leq \vert a_r(t,x)\vert\rho,\quad (t,x)\in{\R}^{n+1}\]
and
\[\frac{\partial \left(\frac{\xi.x}{\tau}\right)}{\partial t}\leq \frac{\vert a_r\vert}{2a}\rho+\frac{\vert a_t\vert}{2a\vert\tau\vert}\vert\xi\vert\rho-1.\]
Condition (8.1) implies 
\[\frac{\partial \left(\frac{\xi.x}{\tau}\right)}{\partial t}\leq-\frac{\rho}{2a}\left(\frac{2a}{\rho}-\frac{\vert a_t\vert}{\sqrt{a}}-\vert a_r\vert\right)\leq  -\left(\frac{\rho}{2a}\left[\frac{2a}{\rho}-\frac{\vert a_t\vert}{\sqrt{\inf a}}-\vert a_r\vert\right]\right),\]
and we obtain
\[\frac{\partial \left(\frac{\xi.x}{\tau}\right)}{\partial t}\leq -\frac{\rho}{2a}\beta\leq-\alpha<0.\]
Therefore, for $t>t_0$ we have 
\[ -2a\frac{x.\xi}{\tau}\geq -2a\frac{x_0.\xi_0}{\tau_0}+2a\alpha (t-t_0)\geq2\alpha C_0(t-t_0)-C_1\vert x_0\vert\]

We deduce
\[\vert x(t)\vert^2\geq \alpha C_0(t-t_0)^2-C_1\vert x_0\vert (t-t_0)+\vert x_0\vert^2\geq  \alpha C_0(t-t_0)^2-C_1\rho (t-t_0)- \rho^2.\]

Thus,  for all $R>\rho$, there exists $T_R>0$ which does not depend on   $(t_0,x_0,\tau_0,\xi_0)$  such that for $(t-t_0)>T_R$, $\vert x(t)\vert>R$.
Since $\frac{\textrm{d}t}{\textrm{d}s}=\tau$, by replacing $\frac{x.\xi}{\tau}$ with $\frac{x.\xi}{-\tau}$, we apply the same argument for $t<t_0$.
Thus, there exists $T_R>0$ which does not depend on  $(t_0,x_0,\tau_0,\xi_0)$  such that  for $\vert x_0\vert\leq\rho$ we have
 $\vert x(\sigma)\vert>R$ as $\vert t(\sigma)-t(0)\vert>T_R$. Consequently,  $a(t,x)$ is a non-trapping metric.

\subsection{Conditions for bounded global energy}
In this subsection we present some examples of metrics $a(t,x)$ such that if $u$ is the solution of (1.1)  for all $t>0$ we have  
\[\Vert (u,u_t)(t)\Vert_{\B}\leq C\Vert f\Vert_{\B},\]
with $C>0$ independent on  $t$,$f$ and $\tau$.
Let $\xi(r)$ be a $C^\infty$ function which depends only on the radius $r=\vert x\vert$ and satisfies the following conditions:
\begin{equation}\label{eq:hyp1}
 \xi''\leq0,\quad 0< r\xi'\leq\xi\leq \epsilon \inf a<\inf a,
\end{equation}
Moreover assume that $\xi$ and $a(t,x)$ satisfy
\begin{equation}\label{eq:hyp1}
 \xi'a-a_t-\xi a_r\geq 0
\end{equation}
and
\begin{equation}\label{eq:hyp1}
 \left(\frac{\xi}{r}-\xi'\right)\left(\frac{a(n-3)}{r}+a_r\right)-a\xi''\geq0.
\end{equation}

\begin{Thm}
Let $a(t,x)$ satisfy conditions (8.3) and (8.4). Then, there exists a constant $C\geq1$ which does not depend on $f,t,s$ such that for each $f\in{\B}$ and all $t\in{\R}$, $s\in{\R}$, $t\geq s$ we have
\[\Vert \mathcal{U}(t,s)f\Vert_{\B}\leq C\Vert f\Vert_{\B}.\]
 \end{Thm}
Let $\xi(r)$ be a $C^\infty$ function which depends only on the radius $r=\vert x\vert$ and satisfies conditions (8.2) and let $a(t,x)$ satisfy conditions (8.3) and (8.4). Set
\[e(u)=\half(a\vert\nabla u\vert^2+u_t^2),\ \ M_\xi(u)=u_t+\xi u_r+\frac{\xi(n-1)}{2r}u,\]
where 
\[\frac{\partial}{\partial r}=\sum_{i=1}^{n}\frac{x_i}{r}\frac{\partial}{\partial x_i},\]
\begin{lem}
Let $u(t,x)\in C^2({\R}^{n+1})$. Then 
\begin{equation} \label{eq:2}
 M_\xi(u)(\partial_t^2 u-\Div_x(a\nabla_x u))=\frac{\partial X}{\partial t}+ \Div_x Y+Z,
\end{equation}
where
\[X=e(u)+\xi u_t(u_r+\frac{n-1}{r}u),\]
\[Y=-a\nabla u M_\xi(u) +\frac{x\xi}{2r}(a\vert\nabla u\vert^2-u_t^2)+\frac{n-1}{4r^2}(\xi'-\frac{\xi}{r})xu^2,\]
\[\begin{array}{lll}Z&=& \half \xi'u_t^2+a(\frac{\xi}{r}-\xi')(\vert\nabla u\vert^2-u_r^2)+\left(\frac{\xi}{r}-\xi'\right)
(\vert\nabla u\vert^2-u_r^2)\\
\ &\ &+\left[\left(\frac{\xi}{r}-\xi'\right)\left(\frac{a(n-3)}{r}+a_r\right)-a\xi''\right]\frac{n-1}{4r}u^2
+\half(a\xi'-a_t-\xi a_r)\vert\nabla u\vert^2.\end{array}\]
\end{lem}
The proof of (8.5) is a direct calculation and we omit it.\\
\textit{Proof of Theorem 8:} It suffices to consider only real-valued solutions of (1.1) with initial data $f\in{\CI}\times{\CI}$. Let $\mathcal{U}(t,s)f=(u(t,x),u_t(t,x))$. For each $t$ the function $u(t,x)$ has compact support with respect to $x$. The conditions (8.2), (8.3), (8.4) imply $Z\geq0$. Therefore integrating (8.2) and assuming $n\geq4$, we get
\[\int_{{\R}^n}\frac{\partial X}{\partial t}\textrm{d}x+\int_{{\R}^n}Z\textrm{d}x+\int_{{\R}^n}\Div(Y)\textrm{d}x=0.\]
On the other hand,
\[\int_{{\R}^n}\Div(Y)\textrm{d}x=\lim_{\delta\to0,R\to\infty} \int_{\delta\leq\vert x\vert\leq R}\Div_x(Y)(x)\textrm{d}x=\lim_{\delta\to0}-\int_{S^{n-1}}\delta^{n-1}Y(\delta x).x\textrm{d}\sigma(x)=0.\]
Also, since $u(t,x)$ has compact support with respect to $x$, we have
\begin{equation}\label{eq:hyp1}\int_{{\R}^n}\frac{\partial X}{\partial t}\textrm{d}x=\frac{\partial\left(\int_{{\R}^n} X\textrm{d}x\right)}{\partial t}.\end{equation}
The equality (8.5) implies
\[\frac{\partial\left(\int_{{\R}^n} X\textrm{d}x\right)}{\partial t}=-\int_{{\R}^n}Z\textrm{d}x\leq0.\]
We obtain
\[\int_{{\R}^n}X(t,x)\textrm{d}x\leq\int_{{\R}^n}X(s,x)\textrm{d}x.\]
For $X$ we use the representation
\[ X=(1-\frac{\xi}{\inf a})e(u)+\xi u^2\frac{(n-1)(n-3)}{8r^2}+\xi'u^2\frac{(n-1)}{4r}\]
\[ \ \ +\Div(-\xi u^2\frac{n-1}{4r^2}x)\]
\[\ \ \ \ \ \ +\xi\left[\frac{e(u)}{\inf a}+\frac{n-1}{r}uu_r+\frac{(n-1)^2}{8r^2}u^2+u_t(u_r+\frac{n-1}{2r}u)\right].\]
Now, since $\inf a\leq1$, it is easy to obtain the inequality
\[u_t(u_r+\frac{n-1}{2r}u)\geq-\half u_t^2-\half\vert u_r+\frac{n-1}{2r}u\vert^2\geq-\left(\frac{e(u)}{\inf a}+\frac{(n-1)}{2r}uu_r+\frac{(n-1)^2}{8r^2}u^2\right).\]
Consequently, the last term in $X$ is non-negative and
\[\int_{{\R}^n}X(t,x)\textrm{d}x\geq\int_{{\R}^n}(1-\frac{\xi}{\inf a})e(u)(t,x)\textrm{d}x\geq\frac{(1-\epsilon)}{2}c\Vert (u,u_t)(t)\Vert_{\B},\]
where $c,C$ are  constants independent of $f$ and $t$ satisfying\\
 for all
 $f=(f_1,f_2)\in{\B}$ and $t>0$
\[c\Vert f\Vert_{{\B}}\leq \int_{{\R}^n}(a(t,x)\vert\nabla f_1\vert^2+(f_2)^2)\textrm{d}x\leq C\Vert f\Vert_{\B}.\]
In the same way, writing $X$ in the form
\[ X=(1+\frac{\xi}{\inf a})e(u)-\xi u^2\frac{(n-1)(n-3)}{8r^2}-\xi'u^2\frac{(n-1)}{4r}\]
\[ \ \ +\Div(\xi u^2\frac{n-1}{4r^2}x)\]
\[\ \ \ \ \ \ -\xi\left[\frac{e(u)}{\inf a}+\frac{n-1}{r}uu_r+\frac{(n-1)^2}{8r^2}u^2+u_t(u_r+\frac{n-1}{2r}u)\right],\]
we conclude that 
\[\int_{{\R}^n}X(s,x)\textrm{d}x\leq\frac{(1+\epsilon)}{2}C\Vert (u,u_t)(s)\Vert_{\B}.\]
Thus, we obtain the estimate
\[\Vert (u,u_t)(t)\Vert_{\B}\leq\frac{C(1+\epsilon)}{c(1-\epsilon)}\Vert (u,u_t)(s)\Vert_{\B}.\]
For $n=3$ the term $Y$ has a singularity at $r=0$ and integrating over $0<\delta\leq\vert x\vert\leq R$, we get
\[\lim_{\delta\to0,\ R\to\infty}\int_{\delta\leq\vert x\vert\leq R}\Div Y\textrm{d}x=\lim_{\delta\to0,\ R\to\infty}\int_{\delta\leq\vert x\vert\leq R}-\Div(\frac{\xi(n-1)}{4r^3}xu^2)\textrm{d}x.\]
Thus, we obtain
\[\int_{{\R}^n}\Div(Y)\textrm{d}x=\frac{n-1}{4}\lim_{\delta\to0}\int_{\mathbb S^2}\delta^2\frac{\xi}{\delta^3}u^2(t,\delta x)\delta x.x\textrm{d}\sigma=\frac{(n-1)\pi}{2}\xi(0)u(t,0)^2\geq0.\]
This expression is non-negative.
Finally, for $t\geq s$ we obtain our result by an approximation with functions with compact support.
$\quad\quad\quad\quad\quad\quad\quad\quad\quad\quad\quad\quad\quad\quad\quad\quad\quad\quad\quad\quad\quad\quad\quad\quad\quad\quad\quad\quad\square$

Notice that with such a metric we have
\[\sigma(Z^b(T,0))\subset\{ z\in\mathbb{C}\ :\ \vert z\vert\leq1\}.\]
and we may have eigenvalues of $Z^b(T,0)$ lying in $\mathbb S^1$. We will find stronger condition on $a(t,x)$ to eliminate eigenvalues  of $Z^b(T,0)$ on $\mathbb S^1$.

\subsection{Exponential decay for the operator $Z$ associate to time periodic non-trapping metric}
The purpose of this subsection is to apply the results for non-trapping metric which does not depend on $t$ to construct a time-periodic non-trapping metric which satisfies condition (H2). Consider a non-trapping metric $a(t,x)$,  $T$-periodic in t with $T>0$ to be determined such that for $T_1\leq t\leq T$,\\
 $a(t,x)=a(T_1,x)$ with $T_1<T$. Set $a_1(x)=a(T_1,x)$. Consider the following problem
\begin{equation} \label{eq:2}  \left\{\begin{array}{c}
v_{tt}-\Div_{x}(a_1(x)\nabla_{x}v)=0,\\
(v,v_{t})(0)=f,\end{array}\right.\end{equation}
and the propagator
\[\mathcal{V}(t):{\B}\ni f\longmapsto (v,v_t)(t)\in{\B}\]
associate to problem (8.7). 

Let $u$ be the solution of the problem (1.1). For $T1\leq t\leq T$ we find
\[\partial_t^2 u -\Div_x(a_1(x)\nabla_x u)=\partial_t^2 u -\Div_x(a(t,x)\nabla_x u)=0.\]
Thus, for all $T_1\leq s<t\leq T$ we have
\[\mathcal{U}(t,s)=\mathcal{V}(t-s)\quad\textrm{and}\quad Z^\rho(T,0)=P_+^\rho\mathcal{V}(T-T_1)\mathcal{U}(T_1,0)P_-^\rho.\]
Proposition 2 implies that $(id_{\B}-P^\rho_-)\mathcal{U}(T_1,0)P^\rho_-=0$,  and we get
\[Z^\rho(T,0)=P^\rho_+\mathcal{V}(T-T_1)P^\rho_-\mathcal{U}(T_1,0)P^\rho_-.\]
Since $a_1(x)$ is a non-trapping metric  which does not depend on $t$ and $n\geq3$ is odd it was established ( see [31] and [32]) that
\[\Vert P^\rho_+\mathcal{V}(t)P^\rho_-\Vert_{\mathcal{L}({\B})}\leq Ce^{-\delta t}\]
 with $C,\delta>0$ independent on $t$. It follows
 \[\Vert Z^\rho(T,0)\Vert_{\mathcal{L}({\B})}\leq Ce^{-\delta(T-T_1)}\Vert \mathcal{U}(T_1,0)\Vert_{\mathcal{L}({\B})}.\]
 Hence, for $T$ large enough we get
 \[r\left(Z^\rho(T,0)\right)\leq\Vert Z^\rho(T,0)\Vert_{\mathcal{L}({\B})}<1\]
 where $r\left(Z^\rho(T,0)\right)$ is the radius spectrum of $Z^\rho(T,0)$.
 For such a metric $a(t,x)$, $Z^\rho(T,0)$ satisfies the condition (H2).

\end{document}